\documentclass{amsart}
\usepackage{geometry}                
\usepackage{graphicx}
\usepackage{amssymb}
\usepackage{epstopdf}
\usepackage{lipsum}
\usepackage{color}
\usepackage[usenames,dvipsnames,svgnames,table]{xcolor}
\usepackage[colorlinks=true, pdfstartview=FitV, linkcolor=blue, citecolor=blue, urlcolor=blue]{hyperref}

\newtheorem{theorem}{Theorem}
\newtheorem{prop}[theorem]{Proposition}
\newtheorem{cor}[theorem]{Corollary}
\newtheorem{lemma}[theorem]{Lemma}

\theoremstyle{definition}
\newtheorem{definition}{Definition}
\newtheorem{example}{Example}
\newtheorem{remark}{Remark}

\newcommand{\idiot}[1]{\vspace{5 mm}\par \noindent
\marginpar{\textsc{Note}}
\framebox{\begin{minipage}[c]{0.95 \textwidth}
#1 \end{minipage}}\vspace{5 mm}\par}

\renewcommand{\idiot}[1]{}

\def\la{{\lambda}}
\def\al{{\alpha}}
\def\be{{\beta}}
\def\ga{{\gamma}}

\def\QQ{{\mathbb Q}}

\def\CC{{\mathbb C}}

\def\st{{\tilde s}}
\def\xt{{\tilde x}}

\def\hht{{\tilde h}}

\def\mt{{\tilde m}}
\def\bP{{\mathcal P}}
\def\bT{{\mathcal T}}
\def\bC{{\mathcal C}}
\def\oP{{\overline\bP}}
\def\oT{{\overline\bT}}
\def\oC{{\overline\bC}}

\def\bfp{{\mathbf p}}
\def\bfpb{{\overline \bfp}}

\def\het{{\tilde{h\!e}}}
\def\mvdash{{\,\vdash\!\!\vdash}}

\def\HEX{{H\!\!\!E}}

\def\db{{\overline d}}

\def\dcl{{\{\!\!\{}}
\def\dcr{{\}\!\!\}}}
\def\o{\overline}

\def\SS{\mathbb{S}}

\newcommand{\pchoose}[2]{\begin{pmatrix}#1\\ #2\end{pmatrix}}

\newdimen\squaresize \squaresize=10pt
\newdimen\thickness \thickness=0.4pt

\def\square#1{\hbox{\vrule width \thickness
     \vbox to \squaresize{\hrule height \thickness\vss
        \hbox to \squaresize{\hss#1\hss}
     \vss\hrule height\thickness}
\unskip\vrule width \thickness}
\kern-\thickness}

\def\vsquare#1{\vbox{\square{$#1$}}\kern-\thickness}

\def\young#1{
\vbox{\smallskip\offinterlineskip
\halign{&\vsquare{##}\cr #1}}}

\def\thisbox#1{\kern-.09ex\fbox{#1}}
\def\downbox#1{\lower1.200em\hbox{#1}}
\newdimen\Squaresize \Squaresize=15pt
\newdimen\Thickness \Thickness=0.4pt

\def\Square#1{\hbox{\vrule width \Thickness
     \vbox to \Squaresize{\hrule height \Thickness\vss
        \hbox to \Squaresize{\hss#1\hss}
     \vss\hrule height\Thickness}
\unskip\vrule width \Thickness}
\kern-\Thickness}

\def\Vsquare#1{\vbox{\Square{$#1$}}\kern-\Thickness}

\title[The Hopf Structure of $S_n$ characters]{The Hopf structure of symmetric group characters
as symmetric functions}
\author{Rosa Orellana and Mike Zabrocki}

\date{}  

\thanks{Work supported by NSF grants DMS-1300512 and DMS-1700058, and by NSERC}



\begin{document}

\begin{abstract}
In \cite{OZ} the authors introduced inhomogeneous bases of the ring of symmetric functions.
The elements in these bases have the property that they evaluate to characters of symmetric groups.
In this article we develop further properties of these bases by proving product and coproduct formulae.
In addition, we give the transition coefficients between
the elementary symmetric functions and the irreducible character basis.
\end{abstract}

\maketitle

\tableofcontents

\section{Introduction}
In \cite{OZ}, the authors introduced a basis of the symmetric functions $\{ \st_\lambda \}$
that specialize to the characters of the irreducible modules of the symmetric group
when the symmetric group $S_n$ is embedded in $GL_n$ as permutation matrices.
This basis provides a new perspective on the representation theory of the symmetric group
which is not well understood.   In addition, several outstanding open problems in combinatorial
representation theory are encoded in the linear algebra related to this basis.

One such problem is the {\it restriction problem} \cite{ButlerKing, King, Lit, Nis, ScharfThibon, STW}. Since
$S_n$ can be viewed as a subgroup of $GL_n$ via the embedding using permutation matrices, it follows that
 any irreducible polynomial module of $GL_n$, with character given by the  Schur function $s_\lambda$,  is a
representation of $S_n$ via this restriction.  Thus, the transition coefficients, $r_{\la,\mu}$,
between the Schur basis $\{s_\lambda\}$ and the irreducible character basis $\{\st_\mu\}$  encode
 the decomposition of a polynomial, irreducible $GL_n$-module into symmetric group irreducibles, i.e.,
 $$s_\lambda = \sum_{\mu} r_{\lambda,\mu} \st_\mu.$$
In addition,  we have shown in \cite{OZ} that the structure coefficient of the  $\{\st_\lambda\}$ basis are the
reduced (or stable) Kronecker coefficients, $\bar{g}_{\alpha,\beta}^\gamma$
\cite{Aizenbud, BO, BDE,  BDO, BOR, BOR1,  EA,  Murg, Murg2, Murg3}.
The history of these problems indicates that they are
difficult and are unlikely to be resolved in a single step.
In order to make further progress, we need to develop some of the properties of
character bases so that they can be treated as familiar objects in the
ring of symmetric functions.

The combinatorics related to the irreducible character basis involves
multiset partitions and multiset tableaux.  Developing
analogues of the notion of `lattice' that exists for words and column
strict tableaux would likely help find a combinatorial interpretations for
\begin{enumerate}
\item the transition coefficients, $r_{\lambda,\mu}$, from the Schur function $s_\lambda$
to the irreducible character basis; and
\item the structure coefficients of the irreducible character basis, $\bar{g}_{\alpha,\beta}^\gamma$.
\end{enumerate}
The Hopf algebra structure that we develop here lays out the combinatorics
and indicates that a combinatorial interpretation for the stable Kronecker and restriction
problems may exist using operations on multiset tableaux.

In order to develop symmetric function expressions for irreducible characters
and to develop their properties, the authors introduced in \cite{OZ} an intermediate
basis $\hht_\lambda$ which represent the characters of the trivial module
in $S_{\lambda_1} \times S_{\lambda_2} \times \cdots \times S_{\lambda_{\ell(\lambda)}}$ induced
to $S_n$.  A third character basis, $\xt_\lambda$, was introduced in \cite{AS}
and it represents the character of the module $\SS^\lambda \otimes \SS^{(n-|\lambda|)}$
of $S_{|\lambda|}\times S_{n-|\lambda|}$ induced to $S_n$, where $\SS^\lambda$ is the irreducible
module indexed by $\lambda$ and $\SS^{(n-|\lambda|)}$ is the trivial module.
 Assaf and Speyer used this basis to show that the  Schur expansion of the $\st_\lambda$ basis is sign alternating in
degree.  We show that the power sum expansion of all three of the character bases is given
by expanding them in terms of two analogues of the power sum basis,
$\bfp_\lambda$ and $\bfpb_\lambda$
(see Equations \eqref{eq:powerrects}--\eqref{eq:pexpansionofcharbases}).

In \cite{OZ2} we gave combinatorial interpretations for some products involving the $\st_\lambda$ basis.
The purpose of this paper is to further develop algebraic properties of the
character and related bases.  The main results in this paper
are product formulae for the $\hht_\lambda$, $\bfp_\lambda$ and $\bfpb_\lambda$ bases; as well as
coproduct formulae for the $\hht_\lambda$, $\bfp_\lambda$, $\bfpb_\lambda$,
$\xt_\lambda$ and $\st_\lambda$ bases.  The coproduct formulae correspond to restriction of characters
from $S_n$ to $S_r\times S_s$  with $r+s=n$.

In \cite{OZ} we provided a combinatorial interpretation for the transition coefficients from the
complete homogeneous basis to the $\st_\lambda$-basis in terms of multiset  tableaux.  In
representation theory, this is equivalent to computing the multiplicities when we restrict
the tensor product of symmetric tensors from $GL_n$ to $S_n$.
In this paper, we give a  combinatorial interpretation for the expansion of the
elementary basis in the irreducible character basis.  This corresponds to finding a combinatorial interpretation
of the multiplicities of  the restriction of the antisymmetric tensor products
from $GL_n$ to $S_n$.

\section{Notation and Preliminaries}
The combinatorial objects that arise in our work are the classical building blocks:
set, multiset, partition, set partition, multiset partition, composition, weak composition,
tableau, words, etc.  In this section we remind the reader about the definitions of these
objects as well as establish notation and usual conventions that we will use in this paper.

A {\it partition} of a non-negative integer $n$ is a weakly decreasing sequence
$\lambda= (\lambda_1, \lambda_2, \ldots, \lambda_\ell)$ such that
$|\lambda|=\lambda_1 + \lambda_2 + \cdots + \lambda_\ell= n$.  We call
$n=|\lambda|$ the {\it size} of $\lambda$.   The $\lambda_i$'s are called
the \emph{parts} of the partition and the number of nonzero parts is called the {\it length} and it is
often denoted $\ell(\lambda)=\ell$.   We use the notation $\lambda\vdash n$ to mean
that $\lambda$ is a partition of $n$.   We reserve $\lambda$ and $\mu$ exclusively for
partitions.  We denote by $m_i(\lambda)$ the \emph{multiplicity} of $i$ in $\lambda$, that is,
the number of times it occurs as a part in $\lambda$.  Another useful notation for a partition
is the exponential notation where $m_i = m_i(\la)$ and
$\la = (1^{m_1}2^{m_2}\cdots k^{m_k})$. With this notation
the number of permutations with cycle structure $\lambda\vdash n$ is
$\frac{n!}{z_\la}$ where
\begin{equation}\label{eq:zeela}
z_\la = \prod_{i=1}^{\la_1} m_i(\la)! i^{m_i(\la)}~.
\end{equation}

The most common operation we use is that of adding a
part of size $n$ to the beginning of a partition.  This
is denoted $(n, \la)$.  If $n < \la_1$, this
sequence will no longer be an integer partition and we will have
to interpret the object appropriately.  Similarly, if
$\lambda$ is non-empty, then let
$\o{\lambda} = (\lambda_2, \lambda_3, \ldots, \lambda_{\ell(\lambda)})$.

For a partition $\la$ the set $\{ (i,j) : 1 \leq i \leq \la_j, 1 \leq j \leq \ell(\la) \}$
are called the \emph{cells} of $\la$ and we represent these cells as stacks of boxes in
the first quadrant with the largest part at the bottom following the `French notation'.
This graphical representation of the cells is called the {\it Young diagram}
of $\lambda$.    A {\it tableau} is a mapping from the set of
cells to a set of labels.   A  tableau will be represented
by filling the boxes of the diagram of a partition with
the labels. In our case, we will encounter tableaux where only a subset of the cells are
mapped to a label.  The  {\it content} of a tableau is the
multiset obtained with the total number of occurrences of each number.

A {\it multiset} is a collection of elements such that repetitions are allowed,  multisets are denoted
by $\dcl b_1, b_2, \ldots, b_r \dcr$.  Multisets will also be represented by
exponential notation so that $\dcl 1^{a_1},2^{a_2},\ldots, \ell^{a_\ell}\dcr$
represents the multiset where the value $i$ occurs
$a_i$ times.

A {\it set partition} of a set $S$ is a set of subsets
$\{ S_1, S_2, \ldots, S_\ell\}$ with $S_i \subseteq S$ for
$1 \leq i \leq \ell$,
$S_i \cap S_j = \emptyset$ for $1 \leq i< j \leq \ell$
and $S_1 \cup S_2 \cup \cdots \cup S_\ell = S$.
A {\it multiset partition} $\pi = \dcl S_1, S_2, \ldots, S_\ell \dcr$
of a multiset $S$ is a similar construction to a
set partition, but now $S_i$ is a multiset, and it is possible that two multisets
$S_i$ and $S_j$ have non-empty intersection (and may
even be equal).  The \emph{length} of a multiset partition is the number of non-empty multisets
in the partition and it is denoted by
$\ell(\pi) = \ell$.  We will use the notation
$\pi \mvdash S$ to indicate that $\pi$ is a multiset
partition of the multiset $S$.

We will use $\mt(\pi)$ to represent the partition of $\ell(\pi)$
consisting of the multiplicities of the multisets which
occur in $\pi$ (e.g. $\mt(\dcl\dcl 1,1,2\dcr,\dcl 1,1,2\dcr,\dcl 1,3\dcr\dcr)
=(2,1)$ because $\dcl 1,1,2\dcr$ occurs 2 times and
$\dcl 1,3\dcr$ occurs 1 time).

For non-negative integers $n$ and $\ell$, a \emph{composition of size} $n$
is an ordered sequence  of positive integers $\alpha = (\alpha_1, \al_2, \ldots, \al_\ell)$
such that
$\al_1+\al_2 + \cdots +\al_\ell=n$.  A \emph{weak composition}
is such a sequence with the condition that $\al_i \geq 0$ (zeros are allowed).
To indicate that $\alpha$ is a composition of $n$ we will
use the notation
$\alpha \models n$ and to indicate that
$\alpha$ is a weak composition of $n$ we will use the notation
$\alpha \models_w n$.  For both
compositions and weak compositions, $\ell(\alpha) := \ell$.

\subsection{The ring of symmetric functions.} \label{sec:symfunc}
For some modern references on this subject see
for example \cite{Mac, Sagan, Stanley, Lascoux}.
The ring of symmetric functions will be denoted
$Sym = \QQ[p_1, p_2, p_3, \ldots]$.  The $p_k$ are
power sum generators and they will be thought of as
functions which can be evaluated at values when
appropriate by making the substitution
$p_k \rightarrow x_1^k + x_2^k + \ldots + x_n^k$
but they will be used algebraically in this ring
without reference to their variables. Thus the degree of each $p_k$ is $k$.
It is well-known that the monomials
in the power sums $p_\lambda := p_{\lambda_1} p_{\lambda_2}
\cdots p_{\lambda_{\ell(\lambda)}}$ where $\lambda \vdash n$
span the subspace of the symmetric functions of degree $n$.

For any symmetric function $f \in Sym$, $f$ is a linear combination of
the power sum basis $f = \sum_\lambda a_\lambda p_\lambda$, then for any expression
$E(x_1, x_2, x_3, \ldots)$ in the set of variables $x_1, x_2, x_3, \ldots$ (and potentially
$y_1, y_2, y_3, \ldots, z_1, z_2, z_3, \ldots$) the notation $f[E]$ will represent
$$f[E(x_1, x_2, x_3, \ldots)] =
\sum_{\lambda} a_{\lambda} \prod_{i=1}^{\ell(\lambda)}
E(x_1^{\lambda_i}, x_2^{\lambda_i}, x_3^{\lambda_i}, \ldots)~.$$
The advantage of this notation is that if $X_n = x_1+x_2+ \cdots +x_n$, then
$f[X_n]$ is a symmetric polynomial in $n$ variables
since $p_k[X_n] = x_1^k + x_2^k + \cdots + x_n^k$ is the power sum symmetric polynomial.
This notation is also practical because by defining
$X = x_1 + x_2 + \cdots$, then $f[X]$ is a symmetric series in the set of variables
$x_1, x_2, x_3, \ldots$ such that if the variables $x_i=0$ for $i>n$, then
the resulting expression is $f[X_n]$.  For a given alphabet $X$, set $Sym_X$ to be
the symmetric functions (or symmetric polynomials if $X$ is finite) in the variables $X$
and $Sym_X \simeq Sym$ as a graded Hopf algebra.

A \emph{Hopf algebra} \cite{Sweedler} is a vector space with both a product and a coproduct that form a bialgebra
structure along with some additional relations including the existence of map known as an antipode.
Hopf algebras arise in many areas of mathematics, but in algebraic combinatorics the symmetric functions
as a graded bialgebra is one of the prototypes for what is known as a `combinatorial Hopf algebra' \cite{GrinbergReiner}.

The symmetric functions have a Hopf structure on it,  \cite[Chapter 2]{GrinbergReiner}, and the coproduct operation
can be encoded in operations of evaluating the symmetric function on sets of
variables.  The $p_k$ are primitive elements of $Sym$ so that $\Delta(p_k) = p_k \otimes
1 + 1 \otimes p_k$.  Since $p_k[X + Y] = p_k[X] + p_k[Y]$, it follows that
if we represent the coproduct of $f$ in Sweedler notation,
$\Delta(f) = \sum_i f^{(i)} \otimes {\tilde f}^{(i)}$, then it is possible to derive,
$$f[X+Y] = \sum_i f^{(i)}[X] {\tilde f}^{(i)}[Y]$$
by linearity relations calculated on the power sum basis.
We will calculate formulas for the coproduct operation on the character bases in
Section \ref{sec:coprods}
and this symmetric function notation will be helpful in our proofs since $Sym \otimes Sym \simeq
Sym_{X,Y}$ as a graded Hopf algebra.

The standard bases of $Sym$ (each indexed by the set of partitions $\lambda$) are
{\it power sum} $\{p_\la\}_{\la}$,
{\it homogeneous/complete} $\{h_\la\}_{\la}$,
{\it elementary} $\{e_\la\}_{\la}$,
{\it monomial} $\{m_\la\}_{\la}$, and
{\it Schur} $\{s_\la\}_{\la}$.
The Hall inner product is defined by declaring that the
power sum basis is orthogonal, i.e.,
$\left< \frac{p_\la}{z_\la}, p_\mu \right>
= \delta_{\la\mu}$, where we use the notation
$\delta_{\lambda\mu} = 1$ if $\lambda=\mu$ and $0$ if $\lambda \neq \mu$.
Under this inner product the Schur functions are orthonormal
and the monomial and homogeneous functions are dual, i.e.,
$\left<s_\lambda,s_\mu \right> = \left<h_\lambda,m_\mu \right> =\delta_{\lambda\mu}$. We use this scalar product to represent values of coefficients by taking scalar products with dual bases.  In particular,
an identity that we will repeatedly use
\begin{equation}\label{eq:dualexpansion}
f = \sum_{\lambda} \left< f, a_\lambda\right> b_\lambda
\end{equation}
for any pair of bases $\{ a_\lambda\}_\lambda$ and $\{ b_\lambda\}_\lambda$
such that $\left< a_\lambda, b_\mu \right> = \delta_{\lambda\mu}$.

The scalar product on $Sym$ naturally extends to a scalar product on $Sym_X, Sym \otimes Sym$
and $Sym_{X,Y}$.  The identities that we will use to calculate coproducts
on the character symmetric functions in Section \ref{sec:coprods}
can be found in \cite[Section 5, p. 91-92]{Mac}.

We will also refer to the irreducible character of the symmetric group indexed by the partition $\la$ and evaluated at a permutation
of cycle structure $\mu$ as the coefficient
$\left< s_\la, p_\mu \right> = \chi^\la(\mu)$.
For $k > 0$, define
$$\Xi_k := 1, e^{2\pi i/k}, e^{4\pi i/k},
\ldots, e^{2(k-1)\pi i/k}$$
as a symbol representing the eigenvalues of a permutation
matrix of a $k$-cycle.  Then for any partition
$\mu$, let
$$\Xi_\mu := \Xi_{\mu_1}, \Xi_{\mu_2},
\ldots, \Xi_{\mu_{\ell(\mu)}}$$
be the multiset of eigenvalues of a permutation matrix with
cycle structure $\mu$.  We will evaluate symmetric
functions at these eigenvalues.  The notation
$f[\Xi_\mu]$ represents taking the element
$f \in Sym$ and replacing $p_k$ in $f$ with
$x_1^k + x_2^k + \cdots + x_{|\mu|}^k$ and then
replacing the variables $x_i$ with the values in $\Xi_\mu$.  In particular,
for $r,k>0$, $p_r[\Xi_k] = 0$ unless either $r=0$ (in which case the value is $1$)
or $k$ divides $r$ (in which case $p_r[\Xi_k] = k$).

From \cite{OZ}, a useful tool that we will repeatedly use
to establish symmetric function identities is the following proposition.
\begin{prop} (\cite[Proposition 38]{OZ} and \cite[Corollary 40]{OZ})
\label{prop:rootsimplsf} Let $f,g \in Sym$ be symmetric functions
of degree less than or equal to some positive integer $n$.
Assume that $$f[\Xi_\ga] = g[\Xi_\ga]$$ for all partitions
$\ga$ such that $|\ga| \leq n$ (respectively, $|\ga| \geq n$), then
$f=g$
as elements of $Sym$.
\end{prop}

\subsection{Symmetric group characters as symmetric functions}
In this section,  we will define five additional bases for the ring of symmetric functions and
discuss their connections to the characters of the symmetric group.
Four of these appeared in \cite{OZ} and the fifth basis was introduced in \cite{AS}.

The following expressions were developed
from character polynomials and yield two analogues of the power sum basis.
For each $i, r>0$, define
\begin{equation}\label{eq:powerrects}
\bfpb_{i^r} = i^r
\left(\frac{1}{i} \sum_{d|i} \mu(i/d) p_d \right)_{\!\!\!r}\hbox{ and }
\bfp_{i^r} = \sum_{k=0}^r (-1)^{r-k} \binom{r}{k}
\bfpb_{i^k}
\end{equation}
where $(x)_k=x(x-1)\cdots(x-k+1)$ denotes the $k$-th falling factorial.

Then if a partition $\gamma$ is expressed as $\gamma = 1^{m_1(\gamma)}2^{m_2(\gamma)}\cdots \ell^{m_\ell(\gamma)}$, we set
$$\bfpb_{\gamma}
:= \prod_{i \geq 1} \bfpb_{i^{m_i(\gamma)}}\hbox{ and }
\bfp_{\gamma}
:= \prod_{i \geq 1} {\mathbf p}_{i^{m_i(\gamma)}}~.
$$

Note that by M\"obius inversion and
Equation \eqref{eq:powerrects} we also have that
\begin{equation}\nonumber
\bfpb_{i^r} = \sum_{k=0}^r \binom{r}{k}
\bfp_{i^k}~.
\end{equation}

Now define three more bases,
\begin{equation}\label{eq:pexpansionofcharbases}
\st_\lambda = \sum_{\gamma \vdash |\lambda|}
\chi^{\lambda}(\gamma) \frac{\bfp_{\ga}}{z_\gamma},
\qquad
\xt_\lambda = \sum_{\gamma \vdash |\lambda|}
\chi^{\lambda}(\gamma) \frac{\bfpb_{\ga}}{z_\gamma}
\quad \hbox{ and } \quad
\hht_\lambda = \sum_{\gamma \vdash |\lambda|} \left< h_\lambda, p_\gamma \right>
\frac{\bfpb_{\ga}}{z_\gamma}~.
\end{equation}

We refer to the set of functions $\st_\lambda$ as the {\it irreducible character basis} since
$\st_\lambda[X_n]$ are symmetric polynomials which evaluate to the irreducible characters
of the symmetric group in the same way that Schur functions evaluate to the
irreducible characters of polynomial $GL_n$ modules.
Assaf and Speyer \cite{AS} used the notation $s^\dagger_\lambda$ for $\st_\la$
and called them `stable Specht polynomials.'
The foundations of these symmetric functions is likely found
in earlier work than that of Specht \cite{Specht} on the representation theory of the symmetric groups
and go back to Frobenius \cite{Frobenius} and Young \cite{Young}.

The symmetric functions $\xt_\nu$ appear implicitly in \cite{AS} as $[ M_\nu ]$.
We refer to the set of functions $\xt_\lambda$ as the {\it induced irreducible character
basis} since, for $n\geq |\lambda|+\lambda_1$,  it encodes the character of the module
$\SS^{\lambda} \otimes \SS^{(n-|\lambda|)}\uparrow_{ S_{|\lambda|}\times S_{n-|\lambda|}}^{S_n}$,
where $\SS^{\lambda}$ is the irreducible module indexed by ${\lambda}$
and $\SS^{(n-|\lambda|)}$ is the trivial module of $S_{n-|\lambda|}$.  As a symmetric function, $\xt_\lambda$
is a sum of irreducible character basis elements with nonnegative coefficients.
We refer to the set of functions $\hht_\lambda$ as the {\it induced trivial character basis}
since the $\hht_\lambda[\Xi_\mu]$ are values of trivial characters induced from Young
subgroups $S_{\lambda_1} \times \cdots \times S_{\lambda_r}$ to the symmetric group $S_n$.

Theorem 14 and Equations (34) and (35) from \cite{OZ} say that if $\lambda$ is a partition
and $n$ is an integer such that $n \geq |\lambda|+\lambda_1$
and $\mu$ is a partition of $n$,
then $\st_\lambda[\Xi_\mu] = \chi^{(n-|\lambda|,\lambda)}(\mu)$.
Proposition 16 and Lemma 15 of \cite{OZ}
says that $\hht_\lambda[\Xi_\mu] = \left< h_{(n-|\lambda|, \lambda)}, p_\mu \right>$
if $n \geq |\lambda|$ and $0$ otherwise.  It is well-known that $\left< h_\mu, p_\gamma \right>
= \sum_{\la \vdash |\mu|} K_{\la\mu} \left< s_\lambda, p_\gamma \right>$, where $K_{\la\mu}$ are
the Kostka coefficients.  Then, it follows that $\hht_\mu = \sum_{\lambda\vdash |\mu|} K_{\la\mu} \xt_\lambda$
and $\xt_\lambda[\Xi_\mu] = \left< s_{(n-|\lambda|)}s_\lambda, p_\mu \right>$
if $n \geq |\lambda|$ and $0$ otherwise.

As Assaf and Speyer \cite[Proposition 5]{AS} point out,
the $\xt_\lambda$ and $\st_\lambda$ bases are related by
\begin{equation}\label{eq:relations}
\xt_\lambda = \sum_{\nu} \st_\nu \quad \hbox{ and } \quad \st_\lambda = \sum_{\nu} (-1)^{|\lambda/\nu|} \xt_\nu
\end{equation}
where the sum on the left
(resp. sum on the right) is over partitions $\nu$ contained in $\lambda$ such that $\lambda/\nu$
is a horizontal strip (resp. vertical strip). That is,  if the Young diagram for $\nu$ is drawn inside the Young diagram
for $\lambda$, then there is at most one cell per column (resp. row) in $\lambda$ which is not in $\nu$.

The $\st_\lambda$ form a basis of the symmetric functions
with the property that $\st_\lambda[X_n]$ is a symmetric polynomial (of inhomogeneous degree)
that evaluates to the characters of the irreducible symmetric group
modules $\SS^{(n-|\lambda|,\lambda)}$.  Then, the character symmetric functions are  a means of encoding the characters of
families of symmetric group modules which are representation stable
(a notion formalized in \cite{ChurchFarb}).  They are  equivalent to finitely generated $FI$-modules,
a category introduced by Church, Ellenberg and Farb
in \cite{ChurchEllenbergFarb} to capture the fundamental properties
of representation theory stability.

Let $FI$ be the category with objects $[n]:=\{1,2,\ldots,n\}$
and whose morphisms are injections $[n] \hookrightarrow [m]$.
An {\it FI-module} is a functor $V$ from $FI$ to the category of $\CC$-modules.
By evaluation at a set $[n]$, the functor $V$ realizes a family of $S_n$ modules $V([n])$.
The character basis is a way of encoding the
character of this family as a symmetric function.

In particular, Theorem 3.3.4 of \cite{ChurchEllenbergFarb} shows that
that if $V$ is a finitely generated $FI$-module, then there a polynomial $P_V(X_1, X_2, \ldots)$
such that the character of a permutation $\sigma$ (of a sufficiently large $n$) acting on
the module $V([n])$ is an evaluation of $P_V$.
This implies by Proposition 12 of \cite{OZ} that there is a corresponding symmetric function $f_V$
such that if the permutation $\sigma$ has cycle structure $\mu$, then the character
of $\sigma$ acting on $V([n])$ is equal to $f_V[\Xi_\mu]$.

Summarizing Theorem 3.3.4 appearing in \cite{ChurchEllenbergFarb} and how they
apply to the irreducible character basis, we have the following proposition.
\begin{prop} For every finitely generated $FI$-module $V$,
there exists a positive integer $M$ and coefficients $a_\lambda$
if and only if for all $n \geq M$,
$$V([n]) \simeq \bigoplus_{\lambda} (\SS^{(n-|\lambda|,\lambda)})^{\oplus a_\lambda}$$
and
$$f_V = \sum_{\lambda} a_\lambda \st_{\la}$$
is a symmetric function such that $f_V[\Xi_\mu]$ is the character
of $V$ evaluated at a permutation $\sigma$
of cycle structure $\mu \vdash n$ acting on $V([n])$.
\end{prop}

\subsection{A scalar product on characters and the Frobenius map}
\label{sec:scalarproduct}
In \cite{OZ2} we introduced a scalar product on symmetric functions
(coming from the usual scalar product on irreducible characters of the symmetric group)
for which the irreducible character basis was orthogonal.  This scalar product
was useful in the application presented in that paper because it allowed us to
calculate an individual coefficient of an irreducible character in
an expression.  In this section we develop some of the properties of this scalar
product further.

For all $n$ sufficiently large,
$$\sum_{\nu \vdash n} \frac{\st_\la[\Xi_\nu] \st_\mu[\Xi_\nu]}{z_\nu}
= \sum_{\nu \vdash n} \frac{\chi^{(n-|\lambda|,\lambda)}(\nu)
\chi^{(n-|\mu|,\mu)}(\nu)}{z_\nu} = \delta_{\lambda\mu}~.$$
Now the right hand side of this expression is independent of $n$. Thus,
for a sufficiently large $n$ and for any symmetric functions $f$ and $g$,
the expression $\sum_{\nu \vdash n} \frac{f[\Xi_\nu] g[\Xi_\nu]}{z_\nu}$
is also independent of $n$ by linearity since  $\{\st_\lambda\}$ is a basis.

Therefore we may define
\begin{equation}\label{eq:scalarproduct}
\left< f, g \right>_@ = \sum_{\nu \vdash n} \frac{f[\Xi_\nu] g[\Xi_\nu]}{z_\nu}
= \frac{1}{n!}\sum_{\sigma \in S_n} f[\Xi_{cyc(\sigma)}] g[\Xi_{cyc(\sigma)}]~.
\end{equation}
for
$n \geq 2 \max(deg(f), deg(g))$ and
where $cyc(\sigma)$ is a partition representing the cycle structure of $\sigma \in S_n$.
We use the $@$-symbol as a subscript of the
right angle bracket to differentiate this scalar product from the
usual scalar product where $\left< s_\lambda, s_\mu \right> = \delta_{\la\mu}$.

We can relate these scalar products by using the Frobenius
map which is a linear isomorphism from the class functions of
the symmetric group to the ring of symmetric functions.
Since we know that characters of the symmetric group (and
hence class functions) can be expressed as symmetric
functions, we can
define the Frobenius map or characteristic map on symmetric
functions
\begin{equation}\label{eq:frob}
\phi_n(f) = \sum_{\nu \vdash n} f[\Xi_\nu] \frac{p_\nu}{z_\nu}~.
\end{equation}
We have that $\phi_n$ is a map from the ring of symmetric functions to the
subspace of symmetric functions of degree $n$. And $\phi_n$ has the property that for symmetric
functions $f$ and $g$,
$$\phi_n(f g) = \phi_n(f) \ast \phi_n(g),$$
where $\ast$ denotes the Kronecker (or internal) product of symmetric functions.
Since
\begin{equation}
\hht_\lambda[\Xi_\nu] = \left< h_{|\nu| - |\lambda|} h_\lambda, p_\nu \right>,\hskip .1in
\xt_\lambda[\Xi_\mu] = \left< s_{(n-|\lambda|)}s_\lambda, p_\mu \right>
\hbox{ and  }\st_\lambda[\Xi_\nu] = \chi^{(|\nu|-|\lambda|,\lambda)}(\nu)
\end{equation}
if $|\nu| \geq |\lambda| + \lambda_1$,
then the image of $\st_\lambda$, $\xt_\lambda$ and $\hht_\lambda$ are
\begin{equation}\label{eq:images}
\phi_n( \hht_\lambda ) = h_{(n - |\la|, \la)},\hskip .1in
\phi_n( \xt_\lambda ) = s_{(n-|\lambda|)} s_\lambda
\hskip .1in\hbox{ and }\hskip .1in
\phi_n( \st_\lambda ) = s_{(n - |\la|, \la)}~.
\end{equation}

\begin{remark}
A result of Solomon \cite[Proposition 3.11]{Sol}, translated into the language of symmetric functions,
shows that the irreducible character of the rook
monoid algebra ${\mathbb C} R_k(n)$ indexed by the partition $\lambda$ (with $|\lambda| \leq k$)
at a partial permutation of cycle type $\mu$ (where $\mu$ is a partition with $|\mu| \leq k$)
is equal to the expression $\chi^\lambda_{R_k(n)}(\mu) =
\left< s_{(|\mu|-|\lambda|)}s_\lambda, p_\mu \right>$.  Since we
have that
\begin{equation*}
\xt_\lambda[\Xi_\mu]
= \left< s_{(|\mu|-|\lambda|)}s_\lambda, p_\mu \right>
= \chi^\lambda_{R_k(n)}(\mu)
\end{equation*}
so then the symmetric functions $\xt_\lambda$ are the {\it irreducible character basis of
the rook monoid} in the same way that the symmetric functions $\st_\lambda$
are the irreducible character basis of the symmetric group.
%
\end{remark}

This leads to the following proposition.
\begin{prop} If $n\geq 2 \max(deg(f),deg(g))$,
then
\begin{equation}\label{eq:twoscalarrelations}
\left< f, g \right>_@ = \left< \phi_n(f), \phi_n(g) \right>
\end{equation}
\end{prop}

\begin{proof} For partitions $\lambda$ and $\mu$,
take an $n$ which is sufficiently large
(take $n \geq 2 \max(|\la|,|\mu|)$), then $(n-|\la|,\la)$
and $(n-|\mu|,\mu)$ are both partitions and this
scalar product can easily be computed on the irreducible character
basis by
\begin{equation}\nonumber
\left<\phi_n(\st_\la), \phi_n(\st_\mu)\right> =
\left<s_{(n-|\la|,\la)}, s_{(n-|\mu|,\mu)}\right> =
\delta_{\la\mu} = \left< \st_\la, \st_\mu \right>_@
~.
\end{equation}
Since this calculation holds on a basis,
Equation \eqref{eq:twoscalarrelations} holds for
all symmetric functions $f$ and $g$ by linearity.
\end{proof}

The $\st_\lambda$ symmetric functions are the orthonormal basis with respect to the scalar product $\left< \cdot, \cdot\right>_@$.
This basis is triangular with respect the Schur basis
($\st_\lambda$ is equal to $s_\lambda$ plus terms of lower degree) and hence it may be calculated using
Gram-Schmidt orthonormalization with respect to the $@$-scalar product.

\begin{remark} Using the fact that $\st_\la$ is an orthonormal basis and 
$s_\lambda = \sum_\gamma r_{\lambda,\gamma} \st_\gamma$ where $r_{\lambda,\gamma}$
are the multiplicities in the decomposition of a polynomial, irreducible $GL_n$-module into symmetric group irreducibles.
We have that
$\left< s_\la, s_\mu \right>_@ = \sum_\gamma  r_{\lambda, \gamma} r_{\mu, \gamma}$
which is a positive integer.
\end{remark}

The analogue $\bfp_\lambda$ of power sums are orthogonal
with respect to the $@$-scalar product in the same way that the power sums
are orthogonal with respect to the Hall scalar product.
By the definition, Equation \eqref{eq:pexpansionofcharbases},
\begin{equation}\label{eq:ptost}
\bfp_\lambda = \sum_{\mu \vdash |\la|}
\left< p_\lambda, s_\mu \right> \st_\mu~.
\end{equation}

\begin{prop}  For all partitions $\lambda$ and $\mu$,
$$\left< \bfp_\lambda, \bfp_\mu \right>_@ = z_\lambda \delta_{\lambda\mu}~.$$
\end{prop}

\begin{proof}
By Equation \eqref{eq:ptost},
$$\left< \bfp_\lambda, \bfp_\mu \right>_@ = \sum_{\gamma \vdash |\lambda|}
\sum_{\nu \vdash |\mu|} \left< p_\lambda, s_\gamma \right>
\left< p_\mu, s_\nu \right> \left< \st_\gamma, \st_\nu \right>_@
= \sum_{\nu \vdash |\mu|} \left< p_\lambda, s_\nu \right>
\left< p_\mu, s_\nu \right>~.$$
This expression is equal to $0$ if $|\lambda|$ is not equal to $|\mu|$.  If they
are equal, the right hand side is equal to $\left< p_\lambda, p_\mu \right>
= z_\lambda \delta_{\lambda\mu}$~.
\end{proof}

\section{Products of character bases}
\label{sec:prods}

In \cite{OZ} we showed that the structure coefficients for the character basis $\st_\lambda$
are the reduced (or stable) Kronecker coefficients, $\bar{g}_{\alpha,\beta}^\gamma$,
$$\st_\alpha \st_\beta = \sum_\gamma \bar{g}_{\alpha,\beta}^\gamma \st_\gamma.$$
One of the main motivations for introducing these  bases and developing their properties
is that it will hopefully lead to a combinatorial interpretation for these coefficients.

In \cite{OZ2} we studied several combinatorial formulae for coefficients of repeated products of
character bases in terms of multiset tableaux satisfying a lattice condition.  In particular, we
gave a combinatorial interpretation for the coefficient of $\st_\nu$ in products of the form

$$\st_{\lambda_1} \st_{\lambda_2} \cdots \st_{\lambda_r} \st_\gamma,$$
where $\gamma$ is any partition and $\lambda_i$'s are positive integers.  Notice that these products
contain the Pieri rule as a special case.  Unfortunately, we do not have the means to
extend this to a combinatorial interpretation for the coefficients $\bar{g}_{\alpha,\beta}^\gamma$.
In order to gain a better understanding we develop product formulae for other bases in this section and
coproduct formulae in Section \ref{sec:coprods}.

\subsection{Products on the power sum bases}
In contrast with the power symmetric function $ p_\lambda$ that is multiplicative,
the new bases $\{\mathbf{p}_\gamma\}$ and $\{\bfpb_\gamma\}$ are not multiplicative
unless the parts of the partition are disjoint.  Therefore, in this section we begin by
describing the structure of products of the form $\bfp_{a^b} \bfp_{a^c}$  and
$\bfpb_{a^b} \bfpb_{a^c}$.

Recall that the $k$-falling factorial, $(x)_k$, has the following product structure
\begin{equation}\nonumber
(x)_n (x)_m = \sum_{k\geq0} \binom mk \binom nk k! (x)_{m+n-k} ~.
\end{equation}
Therefore we can at least easily compute the following
product of one of the power sum bases from this formula.
\begin{prop}  For a positive integers $i, r$ and $s$,
\begin{equation}
\bfpb_{i^r} \bfpb_{i^s} = \sum_{k\geq0}
\binom rk \binom sk k! i^{k} \bfpb_{i^{r+s-k}}~.
\end{equation}
If $m_i(\lambda) = 0$, then
$\bfpb_{i^r} \bfpb_\lambda = \bfpb_{(i^r) \cup \lambda}$.
\end{prop}

It also implies by making a change of basis back and
forth between $\bfp_{i^r}$ to $\bfpb_{i^s}$ that we have
the following expansion for the product of the $\bfp_{i^r}$
elements.
\begin{prop} Define the coefficients
$$c_{i,r,s,a} =\sum_{\ell \geq 0} \sum_{d \geq 0}
\sum_{k \geq 0} (-1)^{r+s-\ell-d} \binom{r}{\ell}
\binom{s}{d} \binom{\ell}{k} \binom{d}{k}
\binom{\ell+d-k}{a} k! i^{k}$$
then
\begin{equation}
\bfp_{i^r} \bfp_{i^s} = \sum_{a=0}^{r+s}
c_{i,r,s,a} \bfp_{i^a}~.
\end{equation}
If $m_i(\lambda) = 0$, then
$\bfp_{i^r} \bfp_\lambda = \bfp_{(i^r) \cup \lambda}$.
\end{prop}

\subsection{Products of induced trivial characters}
A combinatorial interpretation for the
Kronecker product of two complete symmetric functions
expanded in the complete basis is listed as
Exercise 23 (e) in section I.7
of \cite{Mac} and
Exercise 7.84 (b) in \cite{Stanley}.
The earliest reference to this result
that we are aware of is due to Garsia and Remmel \cite{GR}.
The formula for this Kronecker product is
\begin{equation}\label{eq:usualhproduct}
h_{\lambda} \ast h_{\mu}
= \sum_{M} \prod_{i=1}^{\ell(\la)} \prod_{j=1}^{\ell(\mu)} h_{M_{ij}}
\end{equation}
summed over all matrices $M$ of non-negative integers with $\ell(\la)$
rows, $\ell(\mu)$ columns and row sums are given by the vectors
$\la$ and column sums by the vector $\mu$.  For $\lambda, \mu, \nu\vdash n$,
we define $d_{\lambda \mu\nu}$ as the coefficient of $h_\nu$ in $h_\lambda \ast h_\mu$.

Let $\lambda$, $\mu$ and $\nu$ be partitions, then for
$n\geq|\la|+|\mu|+|\nu|$, let $\db_{\la\mu\nu} := d_{(n-|\la|,\la)(n-|\mu|,\mu)(n-|\nu|,\nu)}$ denote
the common (stable) coefficient of $h_{(n-|\nu|,\nu)}$
in $h_{(n-|\la|,\la)} \ast h_{(n-|\mu|,\mu)}$.

\begin{prop}  For partitions $\lambda$ and $\mu$
$$\hht_\la \hht_\mu = \sum_{\nu: |\nu| \leq |\la|+|\mu|}
\db_{\la\mu\nu} \hht_\nu~.$$
\end{prop}
\begin{proof} For each partition $\gamma$
such that $|\gamma|\geq n$ ($n$ sufficiently large), calculation in terms of
characters shows that
\begin{align*}
\hht_\la[\Xi_\ga] \hht_\mu[\Xi_\ga] &=
\left< h_{|\ga|-|\la|} h_\la, p_\ga \right>
\left< h_{|\ga|-|\mu|} h_\mu, p_\ga \right>\\
&= \left< h_{(|\ga|-|\la|,\la)}\ast h_{(|\ga|-|\mu|,\mu)}, p_\ga \right>\\
&= \sum_{|\nu| \leq |\la|+|\mu|}
\db_{\la\mu\nu} \left< h_{(|\ga|-|\nu|,\nu)}, p_\ga \right>\\
&= \sum_{|\nu| \leq |\la|+|\mu|}
\db_{\la\mu\nu} \hht_\nu[\Xi_\ga]~.
\end{align*}
We can conclude by Proposition \ref{prop:rootsimplsf} that the structure
coefficients of the induced trivial characters are the coefficients
$\db_{\la\mu\nu}$.
\end{proof}

Let $S$ be a multiset and $T$ a set.  The restriction of $S$ to
$T$ is the multiset $S|_T = \dcl v \in S : v \in T \dcr$.  Then we can
define the restriction of a multiset partition to the content $T$
by $\pi|_T = \dcl S|_T : S \in \pi \dcr$.  If necessary in this operation
we throw away empty multisets in $\pi|_T$.

We will use the notation $\pi \# \tau$ to represent
a set of multiset partitions
that will appear in the product.
Let $\pi$ and $\tau$ be multiset partitions on disjoint sets $S$ and $T$.
\begin{equation}\label{eq:smashproduct}
\pi \# \tau =
\{ \theta : \theta \mvdash S \cup T, \theta|_{S} = \pi,
\theta|_{T} = \tau \}
\end{equation}
That is, $\theta \in \pi \# \tau$ means that $\theta$ is of the form
\begin{equation}\nonumber
\theta = \dcl S_{i_1}, \ldots, S_{i_{\ell(\pi)-k}},
T_{j_1},\ldots,T_{j_{\ell(\tau)-k}},
S_{i_1'} \cup T_{j_1'}, \ldots, S_{i_k'} \cup T_{j_k'} \dcr
\end{equation}
where $\{ i_1, i_2, \ldots, i_{\ell(\pi)-k}, i_1', i_2', \ldots, i_k' \}
= \{1,2,\ldots,\ell(\pi)\}$
and $\{ j_1, j_2, \ldots, j_{\ell(\tau)-k}, j_1', j_2', \ldots,$ $ j_k' \}
= \{1,2,\ldots,\ell(\tau)\}$.

We propose the following different, but equivalent
combinatorial interpretation for this product of the induced trivial character basis.

\begin{prop} \label{prop:prodht} For multiset partitions
$\pi \mvdash S$ and $\tau \mvdash T$ where the multisets
$S$ and $T$ are disjoint,
$$\hht_{\mt(\pi)} \hht_{\mt(\tau)} = \sum_{\theta \in \pi\#\tau} \hht_{\mt(\theta)}~.$$
\end{prop}

Before we give a proof of this proposition by showing an equivalence with
Equation \eqref{eq:usualhproduct}, we provide an example
to try to clarify any subtleties of the notation.

\begin{example}
Let $\pi = \dcl\dcl1\dcr,\dcl1\dcr,\dcl2\dcr\dcr$ and
$\tau = \dcl\dcl3\dcr,\dcl3\dcr,\dcl4\dcr\dcr$

Below we list the multiset partitions in $\pi \# \tau$
along with the corresponding partition $\mt(\theta)$.
$$\dcl\dcl1\dcr,\dcl1\dcr,\dcl2\dcr,\dcl3\dcr,\dcl3\dcr,\dcl4\dcr\dcr \rightarrow 2211
\hskip .3in\dcl\dcl1,3\dcr,\dcl1\dcr,\dcl2\dcr,\dcl3\dcr,\dcl4\dcr\dcr \rightarrow 11111$$
$$\dcl\dcl1\dcr,\dcl1\dcr,\dcl2,3\dcr,\dcl3\dcr,\dcl4\dcr\dcr \rightarrow 2111
\hskip .3in\dcl\dcl1\dcr,\dcl1,4\dcr,\dcl2\dcr,\dcl3\dcr,\dcl3\dcr\dcr \rightarrow 2111$$
$$\dcl\dcl1\dcr,\dcl1\dcr,\dcl2,4\dcr,\dcl3\dcr,\dcl3\dcr\dcr \rightarrow 221
\hskip .3in\dcl\dcl1,3\dcr,\dcl1,3\dcr,\dcl2\dcr,\dcl4\dcr\dcr \rightarrow 211 $$
$$\dcl\dcl1,3\dcr,\dcl1,4\dcr,\dcl2\dcr,\dcl3\dcr\dcr \rightarrow 1111
\hskip .3in\dcl\dcl1\dcr,\dcl1,3\dcr,\dcl2,3\dcr,\dcl4\dcr\dcr \rightarrow 1111$$
$$\dcl\dcl1\dcr,\dcl1,4\dcr,\dcl2,3\dcr,\dcl3\dcr\dcr \rightarrow 1111
\hskip .3in\dcl\dcl1\dcr,\dcl1,3\dcr,\dcl2,4\dcr,\dcl3\dcr\dcr \rightarrow 1111$$
$$\dcl\dcl1,3\dcr,\dcl1,3\dcr,\dcl2,4\dcr\dcr \rightarrow 21
\hskip .3in\dcl\dcl1,3\dcr,\dcl1,4\dcr,\dcl2,3\dcr\dcr \rightarrow 111$$
As a consequence of Proposition \ref{prop:prodht} we conclude
\begin{equation}\nonumber
\hht_{21} \hht_{21} =
\hht_{111} + 4\hht_{1111} + \hht_{11111}
+ \hht_{21} + \hht_{211} + 2\hht_{2111} + \hht_{221}
+ \hht_{2211}
\end{equation}
or in terms of Kronecker products with $n=8$,
\begin{equation}\nonumber
h_{521} \ast h_{521} =
h_{5111} + 4h_{41111} + h_{311111}
+ h_{521} + h_{4211} + 2h_{32111} + h_{3221}
+ h_{22211}~.
\end{equation}
\end{example}

\begin{proof}
We define a bijection between matrices whose row sums are
$(n-|\la|,\la)$ and whose column sums are $(n-|\mu|,\mu)$
and elements of $\pi \# \tau$ where $\pi$ is a multiset
partition such that $\mt(\pi)$ is $\la$ and $\tau$ is a
multiset partition such that $\mt(\tau) = \mu$.

Let $M$ be such a matrix.  The first row of this matrix
has sum equal to $n-|\la|$ and the sum of row $i$
of this matrix represents the number of times that
some multiset $A$ repeats in $\pi$ (it does not matter what that
multiset is, just that it repeats $\sum_j M_{ij}$ times).
The sum of column $j$ of this matrix (for $j>1$) represents
the number of times that a particular part of the
multiset $B$ repeats in $\tau$ (again, it does not matter the
content of that multiset, just that it is different than
the others).  Therefore the entry $M_{ij}$ is the
number of times that $A \cup B$ repeats in the
multiset $\theta \in \pi \# \tau$.  The value of $M_{i1}$
is equal to the number of times that $A$ appears in
$\theta$ and the value of $M_{1j}$ is the number of times
that $B$ appears in $\theta$.
\end{proof}

\begin{example}
To ensure that the bijection described in the proof is
clear we show the correspondence between some specific multiset
partitions and the non-negative integer matrices to which they correspond.
The second and third row will represent the multiplicities of
$\dcl1\dcr$ and $\dcl2\dcr$ respectively.  The second and third
column will represent the multiplicities of
$\dcl3\dcr$ and $\dcl4\dcr$ respectively.
Rather than consider all multiset partitions,
we will consider only the 4 that we calculated in the last
example that have $\mt(\pi) = 1111.$

$$\dcl\dcl1,3\dcr,\dcl1,4\dcr,\dcl2\dcr,\dcl3\dcr\dcr \leftrightarrow
\begin{bmatrix}n-4&1&0\\0&1&1\\1&0&0\end{bmatrix}$$
$$\dcl\dcl1\dcr,\dcl1,3\dcr,\dcl2,3\dcr,\dcl4\dcr\dcr \leftrightarrow
\begin{bmatrix}n-4&0&1\\1&1&0\\0&1&0\end{bmatrix}$$
$$\dcl\dcl1\dcr,\dcl1,4\dcr,\dcl2,3\dcr,\dcl3\dcr\dcr \leftrightarrow
\begin{bmatrix}n-4&1&0\\1&0&1\\0&1&0\end{bmatrix}$$
$$\dcl\dcl1\dcr,\dcl1,3\dcr,\dcl2,4\dcr,\dcl3\dcr\dcr \leftrightarrow
\begin{bmatrix}n-4&1&0\\1&1&0\\0&0&1\end{bmatrix}$$
\end{example}

The multiset partition notation is therefore not significantly different
than the integer matrices notation, but there are distinct advantages
to an interpretation in terms of multiset partitions.  The main one
is that the notion of multisets in the context of symmetric functions
leads to the combinatorial objects of multiset tableaux
that can be used as a possible object to keep track of stable Kronecker coefficients.

\section{Coproducts of the character bases}
\label{sec:coprods}

The coproduct of symmetric functions corresponds to restriction from $S_n$ to
$S_r \times S_t$ where $r+t=n$.  This is because the coproduct operation
is isomorphic to the operation
of replacing one set of variables with two in the power sum symmetric
function $p_k[X] \rightarrow p_k[X]+p_k[Y]$
and the evaluation of the symmetric function at the eigenvalues
of $S_r \times S_t$ in $S_n$ replaces the $X$ variables by eigenvalues of
the element of $S_r$ and the $Y$ variables by the eigenvalues of the element of $S_t$.
Therefore, formulae involving coproducts of character
basis compute multiplicities for these restrictions.
The main result of this section are coproduct formulae
for the analogues of the power sum bases, the $\hht$, $\st$ and $\xt$ bases.

For a basis $\{ b_\lambda \}$ of the symmetric functions, we
will refer to the {\it coproduct structure coefficients} as the coefficients
of $b_\mu \otimes b_\nu$ in $\Delta( b_\lambda )$ where $\lambda, \mu, \nu$ are all
partitions.  For bases that are of homogeneous degree we have the restriction that
$|\mu| + |\nu| = |\lambda|$, but for the bases here it could be the case that
$|\mu| + |\nu| \leq |\lambda|$.  To summarize the results in this section
we state the following theorem.

\begin{theorem} The coproduct formulae for the character bases are given by:
\begin{itemize}
\item (Theorem \ref{th:coprodst})
For partitions $\lambda, \mu$ and $\nu$ with $|\mu| + |\nu| \leq |\lambda|$,
the coproduct structure coefficients of the basis $\{ \st_\lambda \}$
are $\sum_{\ga} c_{\mu\ga}^\la$
where the sum is over all partitions
$\ga$ such that $\ga/\nu$ is a horizontal strip of size $|\la|-|\mu|-|\nu|$ and
$c_{\mu\ga}^\la$ are the Littlewood-Richardson coefficients.
\item (Propositions \ref{prop:coprodbfpb}, \ref{prop:coprodht}
and Corollary \ref{prop:coprodxt} respectively)
The bases $\{ \bfpb_\lambda \}$, $\{ \hht_\lambda \}$ and $\{ \xt_\lambda \}$
have the same coproduct structure coefficients as the power sum $\{p_\lambda\}$, complete
$\{h_\lambda\}$ and
Schur bases $\{s_\lambda\}$ (respectively).
\item (Proposition \ref{prop:coprodbfp}) For the basis $\{\bfp_\lambda\}$, the second analogue of the power sum basis,
we have
$\Delta(\bfp_{\gamma}) := \prod_{i \geq 1} \Delta(\bfp_{i^{m_i(\gamma)}})$ where
\begin{equation*}
\Delta(\bfp_{i^r})
= \sum_{d=0}^r \binom{r}{d}
\sum_{a=0}^d \binom{d}{a} \bfp_{i^a} \otimes \bfp_{i^{d-a}}~.
\end{equation*}
\end{itemize}
\end{theorem}

\subsection{Coproducts on the power sum bases}
In \cite{OZ} we derived an analog of the Murnaghan-Nakayama rule, see Theorem 20.   This allows us to expand
the power symmetric function in the irreducible character basis.  This is currently  the most efficient method for
computing the irreducible character basis.

For a partition $\lambda$, it follows from the defining relations in
Equation \eqref{eq:pexpansionofcharbases} that
\begin{equation}
\label{eq:barpishtanalogue}
\bfpb_\lambda =
\sum_{\mu \vdash |\la|}
\left< p_\lambda, s_\mu \right> \xt_\mu
= \sum_{\mu \vdash |\la|}
\left< p_\lambda, m_\mu \right> \hht_\mu
~.
\end{equation}

Now this power sum analogue is important because of the following proposition.
\begin{prop}
If $|\mu|<|\la|$, then $\bfpb_\la[\Xi_\mu] = 0$ and if
$|\mu| = |\la|$, then $\bfpb_\la[\Xi_\mu] = \delta_{\la\mu}$.
More generally, $\phi_m(\bfpb_\lambda) = h_{(m - |\lambda|)} p_\lambda$.
\end{prop}

\begin{proof}
We know that $\hht_\lambda[\Xi_\gamma] =
\left< h_{(|\gamma|-|\mu|)} h_\mu, p_\gamma\right>$ (and in
particular the expression is $0$ if $|\gamma|<|\mu|$) by Equation $(6)$ of \cite{OZ}.
By Equation \eqref{eq:barpishtanalogue},
\begin{align*}
\phi_m(\bfpb_\lambda) &= \sum_{\gamma \vdash m}\sum_{\mu \vdash |\la|}
\left< p_\lambda, m_\mu \right> \hht_\mu[\Xi_\gamma] \frac{p_\gamma}{z_\gamma}\\
&= \sum_{\gamma \vdash m}\sum_{\mu \vdash |\la|}
\left< p_\lambda, m_\mu \right> \left< h_{(m-|\mu|)} h_\mu, p_\gamma\right> \frac{p_\gamma}{z_\gamma}
=\sum_{\gamma \vdash m}\left< h_{(m-|\mu|)} p_\lambda, p_\gamma\right> \frac{p_\gamma}{z_\gamma}
\end{align*}
by an application of Equation \eqref{eq:dualexpansion} and the
right hand side of this expression is equal to $h_{(m-|\mu|)} p_\lambda$ by a
second application of Equation \eqref{eq:dualexpansion}.
\end{proof}

We wish to understand as completely as possible the coproduct
structure on these power sum bases.  Since the coproduct is an algebra
homomorphism, it suffices to understand the result on the partitions
of the form $(i^r)$.

\begin{prop}\label{prop:coprodbfpb}  The $\bfpb_\lambda$ basis has the
same coproduct as the power sum basis.  That is,
if the coefficients $a^{\mu}_{\ga\tau}$ are defined
by those appearing in the equation $\Delta(p_\mu) = \sum_{\ga,\tau}
a^{\mu}_{\ga\tau} p_\ga \otimes p_\tau$ then
$\Delta(\bfpb_\mu) = \sum_{\ga,\tau}
a^{\mu}_{\ga\tau} \bfpb_\ga \otimes \bfpb_\tau$
and specifically
\begin{equation}\label{eq:deltabfpb}
\Delta( \bfpb_{i^r} ) = \sum_{k=0}^r \binom{r}{k}
\bfpb_{i^k} \otimes
\bfpb_{i^{r-k}}~.
\end{equation}
\end{prop}

\begin{proof}
Define the coefficients $b_{\al\be}^\la$ appearing in the
coproduct of the
complete symmetric function basis by
\begin{equation}\nonumber
\Delta(h_\la) = \sum_{\al, \be} b^{\la}_{\al\be}
h_\al \otimes h_\be~.
\end{equation}
In Proposition \ref{prop:coprodht}, we will show by direct computation that
$\Delta(\hht_\la) = \sum_{\al, \be} b^{\la}_{\al\be}
\hht_\al \otimes \hht_\be$.
We know already that the coefficient of $p_\ga$
in $h_\lambda$ is equal to the coefficient of
$\bfpb_\ga$ in $\hht_\la$ (and both are
equal to $\left< h_\la, p_\ga/z_\ga \right>$ by Equation \eqref{eq:pexpansionofcharbases}).
Therefore the coefficient of
$\bfpb_\ga \otimes \bfpb_\tau$ in $\Delta(\bfpb_\lambda)$
is equal to
\begin{align*}
\Delta(\bfpb_\lambda) &=
\sum_{\nu} \left< p_\la, m_\nu \right> \Delta(\hht_\nu)
= \sum_{\nu} \sum_{\alpha,\beta}
\left< p_\la, m_\nu \right> b_{\al\be}^\nu
\hht_\al \otimes \hht_\be \\
&= \sum_{\nu} \sum_{\alpha,\beta} \sum_{\ga,\tau}
\left< p_\la, m_\nu \right> b_{\al\be}^\nu
\left< h_\al, p_\ga \right>
\left< h_\be, p_\tau \right>
\frac{\bfpb_\ga}{z_\ga} \otimes \frac{\bfpb_\tau}{z_\tau}~.
\end{align*}
Now if we compute the coefficient of $p_\ga \otimes p_\tau$
in $\Delta(p_\lambda)$ it is precisely the same as the
coefficient of $\bfpb_\ga \otimes \bfpb_\tau$
in $\Delta(\bfpb_\la)$.

In particular, Equation \eqref{eq:deltabfpb} holds because
\begin{equation*}
\Delta(p_{i^r}) = \left( p_i \otimes 1 +
1 \otimes p_i\right)^r = \sum_{k=0}^r
\binom{r}{k}
p_{i^k} \otimes p_{i^{r-k}}~.\qedhere
\end{equation*}
\end{proof}

Now for the coproduct on the other power basis $\bfp_\lambda$ we
need to work slightly harder to give the expression.

\begin{prop} \label{prop:coprodbfp}
For $i,r \geq 1$, the coproduct on $\bfp_\lambda$
can be calculated by
\begin{equation}
\Delta(\bfp_{i^r})
= \sum_{d=0}^r \binom{r}{d}
\sum_{a=0}^d \binom{d}{a} \bfp_{i^a} \otimes \bfp_{i^{d-a}},
\end{equation}
and since $\Delta$ is an algebra homomorphism, $\Delta(\bfp_{\gamma})
:= \prod_{i \geq 1} \Delta(\bfp_{i^{m_i(\gamma)}})$.
\end{prop}

\begin{proof}
Applying Equation \eqref{eq:powerrects} and \eqref{eq:deltabfpb},
we simplify the limits
of the expression to show that
$\binom{r}{k}=0$ if $k>r$.  Then we see that
\begin{align*}
\Delta(\bfp_{i^r}) &= \sum_{k\geq0} (-1)^{r-k} \binom{r}{k}
\Delta(\bfpb_{i^k})\\
&=\sum_{k\geq0} \sum_{\ell\geq0} (-1)^{r-k} \binom{r}{k}
\binom{k}{\ell} \bfpb_{i^\ell} \otimes \bfpb_{i^{k-\ell}}\\
&=\sum_{k\geq0} \sum_{\ell\geq0}
\sum_{a\geq0} \sum_{b\geq 0}
(-1)^{r-k} \binom{r}{k}
\binom{k}{\ell} \binom{\ell}{a} \binom{k-\ell}{b}
\bfp_{i^a} \otimes \bfp_{i^{b}}\\
&=
\sum_{a\geq0} \sum_{b\geq 0}
\sum_{k\geq0} \sum_{\ell\geq0}
(-1)^{r-k} \binom{r}{k}
\binom{k}{\ell} \binom{\ell}{a} \binom{k-\ell}{b}
\bfp_{i^a} \otimes \bfp_{i^{b}}
\end{align*}
where in the last equality we have just rearranged the
order of the summations.
The sum over $a$ and $b$ can be combined by setting $d=a+b$
and changing the sum over $a \geq 0, b\geq 0$ to one of
$d\geq0, a\geq 0$.  Moreover, since $\binom{\ell}{a} = 0$ if $\ell<a$,
then we can assume that $\ell\geq a$ and hence
\begin{align}
\Delta(\bfp_{i^r})&=\sum_{d\geq0} \sum_{a=0}^d
\sum_{k\geq0} \sum_{\ell\geq a}
(-1)^{r-k} \binom{r}{k}
\binom{k}{\ell} \binom{\ell}{a} \binom{k-\ell}{d-a}
\bfp_{i^a} \otimes \bfp_{i^{d-a}}\nonumber\\
&=\sum_{d\geq0} \sum_{a=0}^d
\sum_{k\geq0} \sum_{\ell\geq 0}
(-1)^{r-k} \binom{r}{k}
\binom{k}{\ell+a} \binom{\ell+a}{a} \binom{k-\ell-a}{d-a}
\bfp_{i^a} \otimes \bfp_{i^{d-a}}~.
\label{eq:currentexpr}
\end{align}
Now by expanding the binomial coefficients we know that
$$\binom{k}{\ell+a} \binom{\ell+a}{a} \binom{k-\ell-a}{d-a} =
\binom{d}{a} \binom{k}{\ell+d} \binom{\ell+d}{d}~.$$
Note that by taking the coefficient of $y^{d}$ in
$$\sum_{m=0}^k 2^{k-m} \binom{k}{m}y^m =
(1+1+y)^k = \sum_{m\geq0}\sum_{\ell\geq0} \binom{k}{\ell}\binom{\ell}{m} y^m
= \sum_{m\geq0}\sum_{\ell'\geq0} \binom{k}{\ell'+m}\binom{\ell'+m}{m} y^m$$
we know that
$$
2^{k-d}\binom{k}{d}  = \sum_{\ell \geq 0} \binom{k}{\ell+d} \binom{\ell+d}{d}~.
$$
This reduces Equation  \eqref{eq:currentexpr} to
\begin{align*}
\Delta(\bfp_{i^r}) &=\sum_{d\geq0} \sum_{a=0}^d
\binom{d}{a}
\sum_{k\geq0}
(-1)^{r-k} 2^{k-d} \binom{r}{k} \binom{k}{d}
\bfp_{i^a} \otimes \bfp_{i^{d-a}}\nonumber\\
&=\sum_{d\geq0} \sum_{a=0}^d
\binom{d}{a}
\binom{r}{d}
\bfp_{i^a} \otimes \bfp_{i^{d-a}}~.
\end{align*}
The last equality follows by taking the
coefficient of $y^{d}$ in
$$\sum_{m \geq 0} \binom{r}{m} y^m = (2-1+y)^r =
\sum_{m \geq 0} \sum_{k \geq 0} (-1)^{r-k} 2^{k-m} \binom{r}{k} \binom{k}{m} y^m~.\qedhere$$
\end{proof}

\subsection{Coproducts on the induced trivial and induced irreducible character basis}
The induced trivial character basis follows the combinatorics of
multiset partitions for the product.   We now give the coproduct
for this basis.

It turns out that the coefficient of $\hht_\mu \otimes \hht_\nu$
in $\Delta(\hht_\la)$ is equal to the coefficient $h_\mu \otimes h_\nu$ in $\Delta(h_\la)$.
Since $\Delta(h_n) = \sum_{k=0}^n h_{n-k} \otimes h_k$, we have more generally that
\begin{equation}\nonumber
\Delta(h_\la) = \sum_{\alpha+\beta=\lambda} h_\alpha \otimes h_\beta~.
\end{equation}
Hence the coproduct formula for the induced trivial character basis
can be stated in the following proposition.
\begin{prop} \label{prop:coprodht} For a partition $\lambda$,
\begin{equation}
\Delta(\hht_\la) = \sum_{\alpha+\beta=\lambda} \hht_\alpha \otimes \hht_\beta
\end{equation}
where the sum is over all pairs of weak compositions $\alpha$ and $\beta$ of
length $\lambda$ whose vector sum is equal to $\lambda$.
\end{prop}

\begin{proof}
We will show that $\hht_\lambda[X+Y] = \sum_{\alpha+\beta=\lambda} \hht_\alpha[X] \hht_\beta[Y]$.
This will be done by applying Proposition \ref{prop:rootsimplsf} that states
if $f[\Xi_\mu] = g[\Xi_\mu]$ for enough partitions $\mu$, then $f=g$ as symmetric functions.

We note that for partitions $\mu$ and $\nu$ and a positive integer $n = |\mu|+|\nu|$,
\begin{equation}\nonumber
\hht_\lambda[\Xi_\mu + \Xi_\nu] = \left< h_{(n-|\la|,\la)}, p_\mu p_\nu \right>~.
\end{equation}
Now the evaluation of the,
\begin{align}
\sum_{\alpha+\beta = \lambda} \hht_\alpha[\Xi_\mu] \hht_\beta[ \Xi_\nu] &=
\sum_{\alpha+\beta = \lambda} \left< h_{(|\mu|-|\al|,\al)}, p_\mu \right>
\left<h_{(|\nu|-|\be|,\be)}, p_\nu \right>\nonumber\\
&= \sum_{\alpha+\beta = \lambda}
\sum_{k=0}^{n-|\la|}
\left< h_{(n-|\la|-k,\alpha)}, p_\mu \right>
\left< h_{(k,\be)}, p_\nu \right> \label{eq:coprodready}
\end{align}
where in the last expression all of the terms in sum are assumed
to be $0$ unless $k = |\nu|-|\be|$.  In this case,
$n-k-|\la| = |\mu|+|\nu|-(|\nu|-|\be|) -|\la| = |\mu|-|\al|$.
Now we can recognize the terms in the left entry of the scalar product
as those that arise as the coproduct formula on
the complete basis element $h_{(n-|\la|,\lambda)}$.  If we define
$\left<f \otimes f', g \otimes g' \right> = \left< f, g \right>
\left< f', g' \right>$ then we know (see for instance \cite{Mac}
I.5 p.92 example 25) that
$\left< \Delta(f), g \otimes h \right> = \left< f, gh \right>$.
Hence we have that Equation \eqref{eq:coprodready} is equivalent to
\begin{align*}
\sum_{\alpha+\beta = \lambda} \hht_\alpha[\Xi_\mu] \hht_\beta[ \Xi_\nu]
&= \left< \Delta(h_{(n-|\la|,\lambda)}), p_\nu \otimes p_\mu \right>\nonumber\\
&= \left< h_{(n-|\la|,\lambda)}, p_\nu p_\mu \right>\\
&= \hht_\lambda[\Xi_\mu + \Xi_\nu]~.\nonumber
\end{align*}
Now from Proposition \ref{prop:rootsimplsf} we can conclude
that $\hht_\lambda[X + \Xi_\nu] =
\sum_{\alpha+\beta = \lambda} \hht_\alpha[X] \hht_\beta[ \Xi_\nu]$
as a symmetric function identity and
a second application allows us to conclude that
$\hht_\lambda[X + Y] =
\sum_{\alpha+\beta = \lambda} \hht_\alpha[X] \hht_\beta[Y]$.
\end{proof}

Because the relationship between the $\hht_\lambda$ basis and the $\xt_\lambda$ basis
is the same as the relationship between the $h_\lambda$ and $s_\lambda$ basis,
we can conclude that the coproduct rule on the
induced irreducible character basis will be the same as that for the Schur
basis.

\begin{cor}\label{prop:coprodxt} For a partition $\lambda$,
\begin{equation}
\Delta(\xt_\la) = \sum_{\mu,\nu} c_{\mu\nu}^\lambda \xt_\mu \otimes \xt_\nu
\end{equation}
where the coefficients $c_{\mu\nu}^\lambda$ are the Littlewood-Richardson coefficients,
the same coproduct structure coefficients for the Schur basis.
\end{cor}

\subsection{Coproducts on the irreducible character basis}
The method that we used in the last section to derive
coproduct formula can also be used to derive the coproduct for
the irreducible character basis.  In this case though we reverse
the expression and expand $\st_\lambda[X+Y]$ evaluated at
$X = \Xi_\mu$ and $Y = \Xi_\nu$.

\begin{theorem} \label{th:coprodst}For a partition $\lambda$,
\begin{equation}\label{eq:stcoproduct}
\Delta(\st_\lambda) =
\sum_{\delta: |\delta|\leq|\la|}
\sum_{\beta \vdash |\la|-|\delta|}
\sum_{\eta}
c^\la_{\delta\beta} \st_{\delta} \otimes \st_{\eta}
\end{equation}
where $c^\la_{\delta\beta}$ is the Littlewood-Richardson coefficient
and the inner sum is over partitions $\eta$ such
that the skew partition $\beta/\eta$
is a horizontal strip.

In particular, the coefficient of $\st_\delta \otimes \st_\eta$
in $\Delta(\st_\lambda)$ is equal to
$$\sum_{\beta} c_{\delta\beta}^\la$$
where the sum is over all partitions
$\beta$ such that $\be/\eta$ is a horizontal strip of size $|\la|-|\delta|-|\eta|$.
\end{theorem}

\begin{proof}
Assume that $\mu \vdash N$ where $N$ is ``sufficiently large''.  We can
without loss of generality assume that $N$ is larger than $2|\la|$ because
we will apply Proposition \ref{prop:rootsimplsf}.
We need to show that the following identity holds for all partitions $\mu$, such that $|\mu|>n$ for
some $n$ which is at least as large the degree of the symmetric function.

\begin{align}
\st_\la[\Xi_\mu+ \Xi_\nu] &= \left< s_{(|\mu|+|\nu|-|\la|,\la)}, p_\mu p_\nu \right>\nonumber\\
&= \left< \Delta(s_{(|\mu|+|\nu|-|\la|,\la)}), p_\mu \otimes p_\nu \right>\nonumber\\
&= \sum_{\al\vdash|\mu|}
\left< s_\al, p_\mu \right>
\left< s_{(|\mu|+|\nu|-|\la|,\la)/\al}, p_\nu \right>~.\label{eq:step1}
\end{align}
Now we know that $|\al|=|\mu|>2|\la|$, this implies that $(|\mu|+|\nu|-|\la|,\la)/\al$ is
a skew partition where the first row is disconnected.  The skew Schur function
$s_{(|\mu|+|\nu|-|\la|,\la)/\al}$ is then equal to
$s_{\la/{\overline \al}} \cdot s_{(|\mu|+|\nu|-|\la|-|\al|+|{\overline \al}|)} =
s_{\la/{\overline \al}} \cdot s_{(|\nu|-|\la|+|{\overline \al}|)}$
where $\overline \al$ is the partition $\alpha$ with the first row removed.  Since
we know that $\al \vdash |\mu|$ then $\al$ is determined from $\overline \al$.
In addition,   $\overline \al$ is contained in $\lambda$ for $s_{\la/{\overline \al}}$ to be defined, thus
$|\overline \al|\leq |\lambda|$.
Therefore Equation \eqref{eq:step1} is equivalent to
\begin{align}
\st_\la[\Xi_\mu+ \Xi_\nu]&= \sum_{{\overline \al}:|{\overline \al}|\leq|\lambda|}
\left< s_{(|\mu|-|{\overline \al}|,{\overline \al})}, p_\mu \right>
\left< s_{\la/{\overline \al}} \cdot s_{(|\nu|-|\la|+|{\overline \al}|)}, p_\nu \right>\nonumber\\
&= \sum_{{\overline \al}:|{\overline \al}|\leq|\la|} \sum_{\be \vdash |\la|-|{\overline \al}|}
c^\la_{{\overline \al}\be}
\left< s_{(|\mu|-|{\overline \al}|,{\overline \al})}, p_\mu \right>
\left< s_{\be} \cdot s_{(|\nu|-|\la|+|{\overline \al}|)}, p_\nu \right>\nonumber\\
&= \sum_{{\overline \al}:|{\overline \al}|\leq|\lambda|}
\sum_{\be \vdash |\la|-|{\overline \al}|}
\sum_{\ga}
c^\la_{{\overline \al}\be}
\left< s_{(|\mu|-|{\overline \al}|,{\overline \al})}, p_\mu \right>
\left< s_{\ga}, p_\nu \right>\label{eq:step2}
\end{align}
where the sum over $\gamma$ is of partitions such that
$\ga/\be$ is a horizontal strip of size $|\nu|-|\lambda|+|{\overline \alpha}|$.
But if $\ga/\be$ is a horizontal strip, then $\be/{\overline \ga}$
will also be a horizontal strip.  Moreover we know that since $\ga \vdash |\nu|$
and $\ga$ is determined by ${\overline \ga}$, then Equation \eqref{eq:step2} is equivalent
to
\begin{align*}
\st_\la[\Xi_\mu+ \Xi_\nu] &=\sum_{{\overline \al}:|{\overline \al}|\leq|\la|}
\sum_{\be \vdash |\la|-|{\overline \al}|}
\sum_{{\overline \ga}}
c^\la_{{\overline \al}\be}
\left< s_{(|\mu|-|{\overline \al}|,{\overline \al})}, p_\mu \right>
\left< s_{(|\nu|-|{\overline \ga}|, {\overline \ga})}, p_\nu \right>\\
&= \sum_{{\overline \al}:|{\overline \al}|\leq|\la|}
\sum_{\be \vdash |\la|-|{\overline \al}|}
\sum_{{\overline \ga}}
c^\la_{{\overline \al}\be} \st_{\overline \al}[\Xi_\mu] \st_{\overline \ga}[\Xi_\nu]
\end{align*}
where the inner sums are over partitions ${\overline \ga}$
such that $\be/{\overline \ga}$ is a horizontal strip.
Using the same argument that we did for the induced trivial character
basis at the end of the proof of Proposition \ref{prop:coprodht}, we conclude that
\begin{equation}\label{eq:result}
\Delta(\st_\lambda) =
\sum_{{\overline \al}}
\sum_{\be \vdash |\la|-|{\overline \al}|}
\sum_{{\overline \ga}}
c^\la_{{\overline \al}\be} \st_{\overline \al} \otimes \st_{\overline \ga}~.
\end{equation}
Equation \eqref{eq:result} is equivalent to Equation \eqref{eq:stcoproduct} by setting $\delta=\overline \al$ and $\eta = \overline \ga$.
\end{proof}

\begin{remark}
The \emph{antipode} of a Hopf algebra is part of its defining structure.
In the case of the symmetric functions, the antipode
is an involution $S : Sym \rightarrow Sym$ where $S(s_\lambda) = (-1)^{|\lambda|} s_{\lambda'}$.
The result of Assaf-Speyer \cite{AS} implies $(-1)^{|\lambda|} S( \st_\lambda )$ will be
Schur positive and hence this expression expands positively in the irreducible character basis.
A better understanding of the transition coefficients between the Schur basis and the irreducible character basis
could be used to give a more precise formula for this expression.
\end{remark}

\section{Expansion of the elementary symmetric functions in the irreducible character basis}

In \cite{OZ} we gave the expansion of the complete symmetric function $h_\mu$ in terms of the
irreducible character basis $\st_\la$.   In this section we give the $\st_\la$ expansion of an
elementary symmetric function $e_\mu$.  The proof is similar to the proof of the irreducible character expansion
of a complete symmetric function, see \cite{OZ} Theorem 9.
We will start by proving the expansion in the irreducible character basis except that we
will reserve some of the detailed combinatorial calculations for last.
We assume that there is a total order on the sets that appear in
set partitions and create tableaux to keep track of the terms
in the symmetric function expansion.  The order chosen does not matter, all the matters is
that there is a total order, in the examples we have chosen lexicographic order.
 Let $shape(T)$ be a partition representing the
shape of a tableau $T$ and we again use the overline
notation on a partition to represent the partition with the first
part removed,
${\overline \la} = (\la_2, \la_3, \ldots, \la_{\ell_\la})$.

\begin{theorem}\label{thm:stexpansionofe}
For a partition $\mu$,
\begin{equation}
e_\lambda = \sum_T \st_{\overline{shape(T)}}
\end{equation}
where the sum is over tableaux that are of skew shape $\nu/(\nu_2)$
for some partition $\nu$ and that are weakly increasing in rows and columns
with non-empty sets as labels of the tableaux
(not multisets, that is no repeated values are allowed)
such that the content of the tableau is
$\dcl1^{\la_1},2^{\la_2},\ldots,\ell^{\la_\ell}\dcr$.
A set is allowed to appear multiple times in
the same column if and only if
the set has an odd number of entries. A set is allowed
to appear multiple times in the same row if and only if
the set has an even number of entries.
\end{theorem}

\begin{proof}
In Corollary \ref{cor:elaexpr} we will show that
\begin{equation}\nonumber
e_\la[\Xi_\mu] = \sum_{\pi \vdash \dcl1^{\la_1}, 2^{\la_2},
\ldots, \ell^{(\la_\ell)}\dcr}
\left< h_{n} h_{\mt_e(\pi)} e_{\mt_o(\pi)}, p_\mu \right>~.
\end{equation}
where $\mt_e(\pi)$ (resp. $\mt_o(\pi)$) is the partition representing the multiplicities of
the sets in $\pi$ with an even (resp. odd) number of elements (see Example \ref{Ex:memo}).

For the rest of this proof, we assume that the reader
is familiar with the Pieri rules,  \cite[Chapter 7, p. 339--340]{Stanley}, which state
$h_r s_\lambda$ is the sum of terms $s_\mu$ where the
Young diagram for $\mu$ differs from the Young diagram
of $\lambda$ by adding cells that may occur in the same
row, but not in the same column of the diagram.  Similarly, $e_r s_\lambda$ is the sum of terms $s_\mu$ where the
Young diagram for $\mu$ differs from the Young diagram
of $\lambda$ by adding cells that may occur in the same
column, but not in the same row of the diagram.

Now order the occurrences of the generators in the product
$h_{n} h_{\mt_e(\pi)} e_{\mt_o(\pi)}$ so that they are in the
same order as the total order that is chosen for the sets that
appear in the set partitions.  Using a tableau to keep track of
the terms in the resulting Schur expansion of the product
we will have that
\begin{equation}\nonumber
h_{n} h_{\mt_e(\pi)} e_{\mt_o(\pi)} = \sum_{T} s_{shape(T)}
\end{equation}
where the sum is over tableaux that have $n$ blank cells in
the first row and sets as labels in the rest of the tableau.
Because we multiply by a generator $h_r$ if a set with an
even number of entries occurs $r$ times, then those sets
with an even number of elements can appear multiple times in the
same row, but not in the same column of the tableau.  Similarly,
because we multiply by a generator $e_r$ if a set with an
odd number of entries occurs $r$ times, then those sets
with an odd number of elements can appear multiple times in the
same column, but not in the same row of the tableau.

If the shape of the
tableau is of skew shape $\nu/(n)$ with $n\geq \nu_2$, there are the
same number of tableaux
of skew-shape $(\nu_2 + |{\overline \nu}|, {\overline \nu})/(\nu_2)$
because there is a bijection by deleting blank cells in the first row.

Therefore we have
\begin{equation}\nonumber
e_\la[\Xi_\mu] = \sum_{\pi \vdash \dcl1^{\la_1}, 2^{\la_2},
\ldots, \ell^{\la_\ell}\dcr}
\sum_T
\left< s_{shape(T)}, p_\mu \right> = \sum_{\pi \vdash \dcl1^{\la_1}, 2^{\la_2},
\ldots, \ell^{\la_\ell}\dcr}
\sum_T
\st_{\overline{shape(T)}}[\Xi_\mu]
\end{equation}
where the sum is over those tableau described in the
statement of the proposition.
Our proposition now follows from Proposition \ref{prop:rootsimplsf}.
\end{proof}

\begin{example}
To begin with a small example, consider the expansion of
$e_{21}$.
The following $11$ tableaux follow the
rules outlined in Theorem \ref{thm:stexpansionofe}.
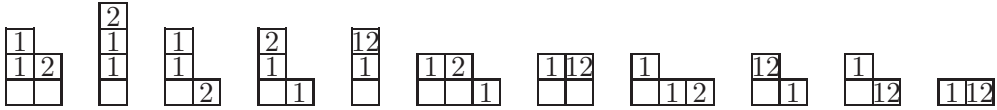
\begin{equation*}
\young{1\cr1&2\cr&\cr}\hskip .2in
\young{2\cr1\cr1\cr\cr}\hskip .2in
\young{1\cr1\cr&2\cr}\hskip .2in
\young{2\cr1\cr&1\cr}\hskip .2in
\young{12\cr1\cr\cr}\hskip .2in
\young{1&2\cr&&1\cr}\hskip .2in
\young{1&12\cr&\cr}\hskip .2in
\young{1\cr&1&2\cr}\hskip .2in
\young{12\cr&1\cr}\hskip .2in
\young{1\cr&12\cr}\hskip .2in
\young{1&12\cr}
\end{equation*}
Theorem \ref{thm:stexpansionofe} then states that
\begin{equation}\nonumber
e_{21} = \st_{21}+\st_{111}+3 \st_{11}+2\st_{2}+3\st_{1}+\st_{()} ~.
\end{equation}
\end{example}

\begin{example}A slightly larger example is the expansion
of $e_{33}$.  If we use Sage \cite{sage,sage-co}
to determine the expansion
we see that
\begin{align*}
e_{33} = &2\st_{()} + 4\st_{1} + 4\st_{11} + 4\st_{111}
+ 4\st_{1111} +
3\st_{11111} + \st_{111111} + 6\st_{2}\\
&+ 8\st_{21} + 7\st_{211} +
4\st_{2111} + \st_{21111} + 5\st_{22} + 4\st_{221}
+ \st_{2211}
+ \st_{222}\\
&+ 5\st_{3} + 4\st_{31} + \st_{311} + \st_{32}
+ \st_{4}~.
\end{align*}
Hence, summing the coefficients,
we see that there are $71$ tableaux in total satisfying the conditions
of Theorem \ref{thm:stexpansionofe} that have a total content
$\dcl1^3,2^3\dcr$.  Listing all $71$ tableaux is
perhaps too large to be a useful example of this theorem,
so let us consider just the
coefficient of $\st_{21}$.  Notice that the content of the tableaux is $\dcl1^3, 2^3\dcr$.
Since the cells of the tableaux must contain sets, there are four ways to partition
$\dcl1^3, 2^3\dcr$ into sets:  $\dcl \{1\},\{1\}, \{1\}, \{2\}, \{2\}, \{2\} \dcr$;
$\dcl \{1\},\{1\}, \{1, 2\}, \{2\}, \{2\} \dcr$;
$\dcl \{1\}, \{1,2\}, \{1, 2\}, \{2\} \dcr$;
and $\dcl \{1,2\},\{1,2\},\{1,2\} \dcr$.   The shapes of the tableaux are of the form $(r, 2,1)/(2)$ where
$r$ depends on the number of boxes filled in the first row.  For the multiset partition $\dcl \{1\},\{1\}, \{1\}, \{2\}, \{2\}, \{2\} \dcr$,
the shape would be $(5,2,1)/(2)$ because we need to fill six boxes, three of these in the first row; however  we cannot
have two equal odd sets in any row,  which means that this partition does not contribute to the coefficient.
Similarly, the partition $\dcl \{1,2\},\{1,2\},\{1,2\} \dcr$
would fill a shape $(2,2,1)/(2)$ (in this case there are no filled boxes in the first row),
but since we cannot repeat sets of even size on any column this partition
does not contribute to the coefficient.

The other two partitions contribute to the coefficient of $\st_{(2,1)}$.
Notice that we have ordered the sets $\{1\}<\{1,2\}<\{2\}$.
The partition $\dcl \{1\},\{1\}, \{1, 2\}, \{2\}, \{2\} \dcr$  will fill tableaux of shape $(4,2,1)/(2)$  and the
partition  $\dcl \{1\}, \{1,2\}, \{1, 2\}, \{2\} \dcr$ will fill tableaux of shape $(3,2,1)/(2)$.  Below we list the
$8$ tableaux that contribute to the coefficient.
\begin{equation*}
\young{1\cr1&2\cr&&12&2\cr}\hskip .2in
\young{12\cr1&2\cr&&1&2\cr}\hskip .2in
\young{2\cr1&12\cr&&1&2\cr}\hskip .2in
\young{2\cr1&2\cr&&1&12\cr}
\end{equation*}
\begin{equation*}
\young{2\cr12&12\cr&&1\cr}\hskip .2in
\young{12\cr1&2\cr&&12\cr}\hskip .2in
\young{2\cr1&12\cr&&12\cr}\hskip .2in
\young{12\cr1&12\cr&&2\cr}
\end{equation*}
\end{example}

In  \cite[Lemma 5.10.1]{Lascoux}, Lascoux showed the following result.  This result
will serve as the starting point for our computations.
In the following expressions, the notation $r|n$ indicates shorthand
for ``$r$ divides $n$.''

\begin{prop} \label{prop:evalhnXr}
For $r \geq 0$, $h_0[\Xi_r] = e_0[\Xi_r]=p_0[\Xi_r]=1$.
In addition, for $n>0$,
\begin{equation}
h_n[\Xi_r] = \delta_{r|n}, \hskip .2in
p_n[\Xi_r] = r \delta_{r|n}, \hskip .2in
e_n[\Xi_r] = (-1)^{r-1}\delta_{r=n}~.
\end{equation}
\end{prop}

We will need the evaluation and a combinatorial interpretation
of $e_\lambda[\Xi_\mu]$ in order to make a
connection with character symmetric functions.

To extend this further, we determine the evaluation of
an elementary symmetric function at $\Xi_\mu$.
For a subset $S = \{i_1, i_2, \ldots, i_{|S|}\} \subseteq
\{1,2,\ldots,\ell(\mu)\}$, let $\mu_S$
denote the sub-partition $(\mu_{i_1}, \mu_{i_2}, \ldots, \mu_{i_{|S|}})$.
This implies that
\begin{equation}
\label{eq:enatXimu}
e_n[\Xi_\mu] = \sum_{\substack{\alpha \models_w n\\
\ell(\alpha) = \ell(\mu)}}
\prod_{i=1}^{\ell(\mu)} e_{\alpha_i}[\Xi_{\mu_i}]
=
\sum_{S : |\mu_S| = n}
\prod_{i \in S} e_{\mu_i}[\Xi_{\mu_i}]
= \sum_{S : |\mu_S| = n} (-1)^{n+|S|}
\end{equation}
where the sum is over all subsets $S \subseteq \{1,2,\ldots,\ell(\mu)\}$
such that $|\mu_S| = n$.

\begin{definition}
Define the set $\oC_{\la,\mu}$ to be the set of sequences
$( S^{(1)}, S^{(2)}, \ldots, S^{(\ell(\la))})$
where each $S^{(i)}$ is a subset such that
$|\mu_{S^{(i)}}| = \la_i$.
\end{definition}

Since $e_\lambda[\Xi_\mu] =
e_{\la_1}[\Xi_\mu]e_{\la_2}[\Xi_\mu] \cdots e_{\la_{\ell(\la)}}[\Xi_\mu]$,
it implies that we have the following Proposition for evaluating
this expression.
\begin{prop}For partitions $\la$ and $\mu$,
\begin{equation}
e_\lambda[\Xi_\mu] =
\sum_{S^{(\ast)} \in \oC_{\lambda,\mu}} (-1)^{|\la|+|S^{(\ast)}|}
\end{equation}
where $|S^{(\ast)}| = \sum_{i=1}^{\ell(\la)} |S^{(i)}|$.
\end{prop}

\begin{example} Let $n=4$, then to evaluate
$e_4[\Xi_{3211}]$ there are three subsets of parts of $(3,2,1,1)$
which sum to $4$, namely, $\{1,3\}, \{1,4\}$ and $\{2,3,4\}$.
The first two are counted with weight $1$ and the third has weight $-1$,
hence $e_4[\Xi_{3211}] = 1+1-1=1$.

To evaluate $e_{31}[\Xi_{3211}]$ we determine that
$\oC_{31,3211} = \{(\{1\},\{3\})$, $(\{1\},\{4\}),$
$(\{2,3\},\{3\}),$ $(\{2,3\},\{4\}),$
$(\{2,4\},\{3\})$, $(\{2,4\},\{4\})\}$.
The first two
of these have weight $(-1)^{|\la|+|S^{(\ast)}|}$ both equal to $1$
and the last four have weight $-1$ hence $e_{31}[\Xi_{3211}] = -2$.
\end{example}

Let $\gamma^{(\ast)} = (\ga^{(0)}, \ga^{(1)}, \ga^{(2)}, \ldots, \ga^{(r)})$
be sequences of partitions such that $\bigcup_{i=0}^r \ga^{(i)} = \mu$,
then we may use Equation \eqref{eq:zeela} to compute
\begin{align}
\frac{z_\mu}{z_{\ga^{(0)}} z_{\ga^{(1)}} z_{\ga^{(2)}}
\cdots z_{\ga^{(r)}}}
&= \prod_{i \geq 1}
\pchoose{m_i(\mu)}{
m_i(\ga^{(1)}),
m_i(\ga^{(2)}), \ldots, m_i(\ga^{(r)})}\label{eq:multiptnchoose}
\end{align}
and the parts of $\ga^{(0)}$ are determined from $\mu$ and all of the
$\ga^{(i)}$ for $1 \leq i \leq r$.

Now we will need to evaluate
$\HEX_{(\la|\tau),\mu} :=
\left< p_\mu, h_{|\mu|-|\la| - |\tau|} h_\lambda e_\tau\right>$
where $\la, \tau$ and $\mu$ are partitions.
We do this by expanding the expression $h_{|\mu|-|\la| - |\tau|} h_\lambda e_\tau$.
For each sequence of partitions $\ga^{(\ast)}$ of length $\ell(\la)$ and
each sequence of partitions $\nu^{(\ast)}$ of length $\ell(\tau)$, there
will be one term in the sum coming from the expansion of the product of $h_\lambda$ and $e_\tau$.
Expanding the expression for $\HEX_{(\la|\tau),\mu}$, yields
\begin{equation*}
\HEX_{(\la|\tau),\mu} =
\sum_{\ga^{(\ast)},\nu^{(\ast)}}
sgn(\nu^{(\ast)})
\prod_{i=1}^{\mu_1}
\pchoose{m_i(\mu)}{m_i(\ga^{(1)}),
\ldots,m_i(\ga^{(\ell(\la))}),m_i(\nu^{(1)}),
\ldots,m_i(\nu^{(\ell(\tau))})}
\end{equation*}
where the sum is over all sequences
of partitions $\ga^{(\ast)} =
(\ga^{(1)}, \ga^{(2)}, \ldots, \ga^{(\ell(\la))})$
where $\ga^{(j)} \vdash \lambda_j$ and
$\nu^{(\ast)} = (\nu^{(1)}, \nu^{(2)}, \ldots, \nu^{(\ell(\tau))})$
where $\nu^{(j)} \vdash \tau_j$ and
\begin{equation}\nonumber
sgn(\nu^{(\ast)}) = (-1)^{\sum_i |\nu^{(i)}|+\ell(\nu^{(i)})}~.
\end{equation}
Note that we are using the convention that
if $\bigcup_i \ga^{(i)} \cup \bigcup_i \nu^{(i)}$ is
not a subset of the parts of $\mu$ then the weight
$$\prod_{i=1}^{\mu_1}
\pchoose{m_i(\mu)}{m_i(\ga^{(1)}),
\ldots,m_i(\ga^{(\ell(\la))}),m_i(\nu^{(1)}),
\ldots,m_i(\nu^{(\ell(\tau))})}$$
is equal to $0$.

\begin{prop}\label{prop:evalhet}
For partitions $\la, \tau$ and $\mu$, let
${\mathcal F}^\mu_{\lambda,\tau}$ be the fillings of the diagram
for the partition $\mu$ with $\la_i$ labels $i$ and
$\tau_j$ labels $j'$ such that all cells in a row are filled with
the same label such that cells in any row are either all filled or all empty.
For $F \in {\mathcal F}^\mu_{\lambda,\tau}$,
the weight of the filling, $wt(F)$ is equal to $-1$ raised to the number
of cells filled with primed labels plus the number of rows occupied
by the primed labels.  Then
\begin{equation}
\HEX_{(\lambda|\tau),\mu}
= \sum_{F \in {\mathcal F}^\mu_{\lambda,\tau}}
wt(F)~.
\end{equation}
\end{prop}

\begin{proof} This is precisely the analogous statement to
Proposition 28 of \cite{OZ}.  It follows because if we fix
the sequences of partitions $\ga^{(\ast)}$ and
$\nu^{(\ast)}$ such that $\ga^{(i)} \vdash \la_i$
and $\nu^{(j)}\vdash \tau_j$, then the quantity
\begin{equation}\nonumber
\prod_{i=1}^{\mu_1} \pchoose{m_i(\mu)}{m_i(\ga^{(1)}),
\ldots,m_i(\ga^{(\ell(\la))}),m_i(\nu^{(1)}),
\ldots,m_i(\nu^{(\ell(\tau))})}
\end{equation}
is precisely the number of $F$ in ${\mathcal F}^\mu_{\lambda,\tau}$
with the rows filled according to the sequences of partitions
$\ga^{(\ast)}$ and $\nu^{(\ast)}$.  The sign of a filling is
constant on this set and is equal to $sgn(\nu^{(\ast)})$.
\end{proof}

\begin{example}The following are all the possible
fillings of the diagram $(3,3,2,2,1,1)$ with two $1$'s and two $1'$'s
such that the rows have the same labels.
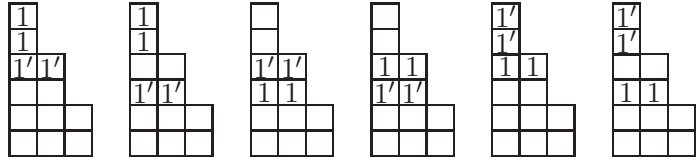
\begin{equation*}
\young{1\cr1\cr1'&1'\cr&\cr&&\cr&&\cr}\hskip .2in
\young{1\cr1\cr&\cr1'&1'\cr&&\cr&&\cr}\hskip .2in
\young{\cr\cr1'&1'\cr1&1\cr&&\cr&&\cr}\hskip .2in
\young{\cr\cr1&1\cr1'&1'\cr&&\cr&&\cr}\hskip .2in
\young{1'\cr1'\cr1&1\cr&\cr&&\cr&&\cr}\hskip .2in
\young{1'\cr1'\cr&\cr1&1\cr&&\cr&&\cr}
\end{equation*}
Since the weight of the filling is equal to the
$(-1)$ raised to the number of cells plus the number of
rows occupied by primed entries, the first four have weight
$-1$ and the last two have weight $1$ and hence
\begin{equation}\nonumber
\HEX_{(2|2),332211} = -2~.
\end{equation}
\end{example}

The rest of this section develops the combinatorial constructions
required to show that the evaluations of $e_\lambda$ at roots
of unity are correct.

We have previously used the notation
$\pi \mvdash \dcl1^{\la_1},2^{\la_2}\ldots,\ell^{\la_\ell}\dcr$
to indicate that $\pi$ is a multiset partition of a multiset.
We will then use the notation
$\pi \vdash \dcl1^{\la_1},2^{\la_2}\ldots,\ell^{\la_\ell}\dcr$
to indicate that $\pi$ is a {\it set partition of a multiset}, that is,
$\pi = \dcl P^{(1)}, P^{(2)}, \ldots, P^{(\ell(\pi))}\dcr$ where
$P^{(1)}\uplus P^{(2)}\uplus \cdots \uplus P^{(\ell(\pi))} =
\dcl1^{\la_1},2^{\la_2}\ldots,\ell^{\la_\ell}\dcr$ and
each of the $P^{(i)}$ are sets (no repetitions allowed).
It is possible that $\pi$ itself is a multiset since it is possible that
$P^{(i)} = P^{(j)}$ when $i \neq j$.
In this case we say that $\pi$ is a set partition of a
multiset of content $\la$ (to differentiate from a multiset partition of
a multiset).

Now we have previously defined $\mt(\pi)$ to be a partition
representing the multiplicity of the sets that appear in
$\pi$.  Now define $\mt_e(\pi)$ be a partition representing
the multiplicities of the sets with an even number of elements
and $\mt_o(\pi)$ be a partition representing the multiplicities
of the sets with an odd number of elements.

\begin{example}
\label{Ex:memo}
Let $\lambda = (5,3,3,2,1)$ and then
\begin{equation}\nonumber
\pi = \dcl \{1,2,5\},\{1,2\},\{1,2\},\{1,3\},\{1,3\},\{3,4\},\{4\}\dcr
\end{equation}
is a set partition of the multiset $\dcl1^5,2^3,3^3,4^2,5\dcr$.
The corresponding partition $\mt(\pi) = (2,2,1,1,1)$ and
$\mt_e(\pi) = (2,2,1)$ and $\mt_o(\pi)=(1,1)$.
The sequence $\mt(\pi)$ is a partition of the length of $\pi$
and $\mt_e(\pi) \cup \mt_o(\pi) = \mt(\pi)$.
\end{example}

\begin{definition}For partitions $\la$ and $\mu$, let
$\oP_{\la\mu}$ be the set of pairs $(\pi, T)$ where $\pi$
is a set partition of the multiset
$\dcl1^{\la_1},2^{\la_2},\ldots,\ell(\la)^{\la_{\ell(\la)}}\dcr$
and $T$ is a filling of some of the rows of the diagram for
$\mu$ with content $\ga = \mt_e(\pi)$ consisting of labels
$\dcl1^{\ga_1},2^{\ga_2},\cdots,\ell(\ga)^{\ga_{\ell(\ga)}}\dcr$
and some rows filled with content
$\tau = \mt_o(\pi)$ consisting of primed labels
$\dcl{1'}^{\tau_1},{2'}^{\tau_2},\ldots,{\ell(\tau)'}^{\ell(\tau)}\dcr$.
The {\it weight}, $wt$, of a pair $(\pi,T)$ will be either $\pm1$
and is equal to $-1$ raised to the number of primed labels
plus the number of rows those labels occupy.
\end{definition}

\begin{example}\label{ex:P31a332211}
Consider the set $\oP_{(3,1),(3,3,2,2,1,1)}$ that consists
of the following 12 pairs of set partitions and fillings
{\small
\squaresize=9pt
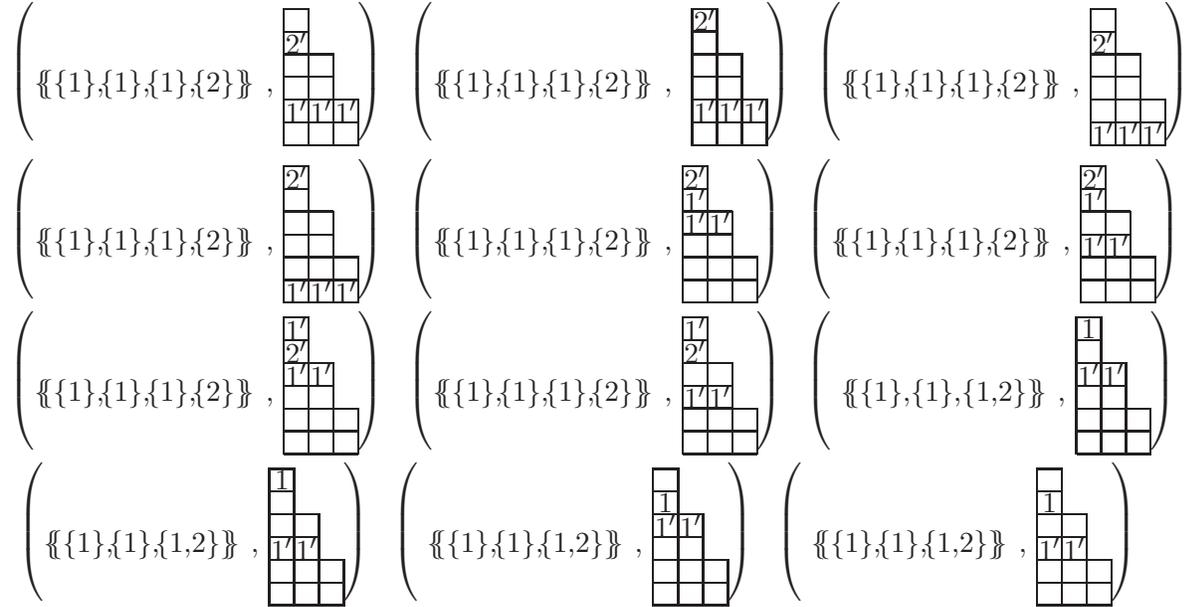
\begin{equation*}
\left(\raisebox{-5pt}{\dcl\{1\},\!\{1\},\!\{1\},\!\{2\}\dcr\ ,}
\raisebox{-25pt}{\young{\cr2'\cr&\cr&\cr1'&1'&1'\cr&&\cr}}\right)\hskip .1in
\left(\raisebox{-5pt}{\dcl\{1\},\!\{1\},\!\{1\},\!\{2\}\dcr\ , }
\raisebox{-25pt}{\young{2'\cr\cr&\cr&\cr1'&1'&1'\cr&&\cr}}\right)\hskip .1in
\left(\raisebox{-5pt}{\dcl\{1\},\!\{1\},\!\{1\},\!\{2\}\dcr\ ,}
\raisebox{-25pt}{\young{\cr2'\cr&\cr&\cr&&\cr1'&1'&1'\cr}}\right)
\end{equation*}
\begin{equation*}
\left(\raisebox{-5pt}{\dcl\{1\},\!\{1\},\!\{1\},\!\{2\}\dcr\ ,}
\raisebox{-25pt}{\young{2'\cr\cr&\cr&\cr&&\cr1'&1'&1'\cr}}\right)\hskip .1in
\left(\raisebox{-5pt}{\dcl\{1\},\!\{1\},\!\{1\},\!\{2\}\dcr\ ,}
\raisebox{-25pt}{\young{2'\cr1'\cr1'&1'\cr&\cr&&\cr&&\cr}}\right)\hskip .1in
\left(\raisebox{-5pt}{\dcl\{1\},\!\{1\},\!\{1\},\!\{2\}\dcr\ ,}
\raisebox{-25pt}{\young{2'\cr1'\cr&\cr1'&1'\cr&&\cr&&\cr}}\right)
\end{equation*}
\begin{equation*}
\left(\raisebox{-5pt}{\dcl\{1\},\!\{1\},\!\{1\},\!\{2\}\dcr \ ,}
\raisebox{-25pt}{\young{1'\cr2'\cr1'&1'\cr&\cr&&\cr&&\cr}}\right)\hskip .1in
\left(\raisebox{-5pt}{\dcl\{1\},\!\{1\},\!\{1\},\!\{2\}\dcr\ ,}
\raisebox{-25pt}{\young{1'\cr2'\cr&\cr1'&1'\cr&&\cr&&\cr}}\right)\hskip .1in
\left(\raisebox{-5pt}{ \dcl\{1\},\{1\},\{1,2\}\dcr\ ,}
\raisebox{-25pt}{\young{1\cr\cr1'&1'\cr&\cr&&\cr&&\cr}}\right)
\end{equation*}
\begin{equation*}
\left(\raisebox{-5pt}{\dcl\{1\},\!\{1\},\!\{1,2\}\dcr\ ,}
\raisebox{-25pt}{\young{1\cr\cr&\cr1'&1'\cr&&\cr&&\cr}}\right)\hskip .1in
\left(\raisebox{-5pt}{ \dcl\{1\},\!\{1\},\!\{1,2\}\dcr\  ,}
\raisebox{-25pt}{\young{\cr1\cr1'&1'\cr&\cr&&\cr&&\cr}}\right)\hskip .1in
\left(\raisebox{-5pt}{ \dcl\{1\},\!\{1\},\!\{1,2\}\dcr\  ,}
\raisebox{-25pt}{\young{\cr1\cr&\cr1'&1'\cr&&\cr&&\cr}}\right)
\end{equation*}
}
The first four of these pairs have weight $+1$ and
the remaining eight have weight $-1$.
\end{example}

With these definitions, we can
use Proposition \ref{prop:evalhet} to state
that
\begin{equation}\nonumber
\sum_{\pi \vdash \dcl1^{\la_1}, 2^{\la_2},
\ldots, \ell^{(\la_\ell)}\dcr}
\HEX_{(\mt_e(\pi)|\mt_o(\pi)),\mu} = \sum_{F \in \oP_{\la,\mu}} wt(F)~.
\end{equation}

Next we define a set  $\oT_{\la,\mu}$, this set is defined in a similar way
as $\oP_{\la,\mu}$, with the main difference that now the tableaux
will contain the sets that make up the parts of $\pi$.

\begin{definition}For partitions $\la$ and $\mu$ let
$\oT_{\la,\mu}$ be the fillings of some of the cells
of the diagram of the partition $\mu$ with subsets
of $\{1,2,\ldots,\ell(\la)\}$ such that the total
content of the filling is
$\dcl1^{\la_1},2^{\la_2},\ldots,\ell(\la)^{\la_{\ell(\la)}}\dcr$
and such that all cells in the same row have the same
subset of entries.  We will define the weight of one
of these fillings to be $-1$ to the power of the size of
$\la$ plus the number of rows whose cells are occupied by
a set of odd size (this is also equal to the number of cells
plus the number of rows occupied by the sets of odd size).
\end{definition}

\begin{example}The following $12$ tableaux are the
elements of $\oT_{(3,1),(3,3,2,2,1,1)}$.
\begin{equation*}
\young{\cr2\cr&\cr&\cr1&1&1\cr&&\cr}\hskip .2in
\young{2\cr\cr&\cr&\cr1&1&1\cr&&\cr}\hskip .2in
\young{\cr2\cr&\cr&\cr&&\cr1&1&1\cr}\hskip .2in
\young{2\cr\cr&\cr&\cr&&\cr1&1&1\cr}\hskip .2in
\young{2\cr1\cr1&1\cr&\cr&&\cr&&\cr}\hskip .2in
\young{2\cr1\cr&\cr1&1\cr&&\cr&&\cr}
\end{equation*}
\begin{equation*}
\young{1\cr2\cr1&1\cr&\cr&&\cr&&\cr}\hskip .2in
\young{1\cr2\cr&\cr1&1\cr&&\cr&&\cr}\hskip .2in
\young{12\cr\cr1&1\cr&\cr&&\cr&&\cr}\hskip .2in
\young{12\cr\cr&\cr1&1\cr&&\cr&&\cr}\hskip .2in
\young{\cr12\cr1&1\cr&\cr&&\cr&&\cr}\hskip .2in
\young{\cr12\cr&\cr1&1\cr&&\cr&&\cr}
\end{equation*}
The fillings listed above are in the same order
as their isomorphism with the set of pairs $\oP_{(3,1),(3,3,2,2,1,1)}$
from Example \ref{ex:P31a332211}.
As in that case we have that,
the first four of these pairs have weight $+1$ and
the remaining eight have weight $-1$.
\end{example}

The following result should be clear from
the definitions listed above and the examples
we have presented.
\begin{lemma} There is a bijection between
the sets $\oP_{\la,\mu}$ and $\oT_{\la,\mu}$ and
$\oC_{\la,\mu}$
that preserves the weight.
\end{lemma}

\begin{cor} \label{cor:elaexpr} For partitions $\la$ and $\mu$,
\begin{equation}
e_\la[\Xi_\mu] = \sum_{\pi \vdash \dcl1^{\la_1}, 2^{\la_2},
\ldots, \ell^{\la_\ell}\dcr}
\HEX_{(\mt_e(\pi)|\mt_o(\pi)),\mu}
\end{equation}
\end{cor}

\begin{example}
There are only three set partitions of $\dcl1^3,2^2\dcr$.  These are
\begin{equation*}
\dcl\{1\},\{1\},\{1\},\{2\},\{2\}\dcr,
\dcl\{1\},\{1\},\{1,2\},\{2\}\dcr,
\dcl\{1\},\{1,2\},\{1,2\}\dcr~.
\end{equation*}
Corollary \ref{cor:elaexpr} states that
\begin{equation}\nonumber
e_{32}[\Xi_\mu] = \HEX_{(\cdot|32),\mu} + \HEX_{(1|21),\mu} + \HEX_{(2|1),\mu}
\end{equation}
\end{example}

\begin{remark}
The expressions $\HEX_{(\lambda,\tau),\mu}$ implies
that we could define symmetric functions
$\het_{(\lambda|\tau)}$ with the property
$\het_{(\lambda|\tau)}[\Xi_\mu] =
\HEX_{(\la|\tau),\mu} =
\left< h_{|\mu|-|\la|-|\tau|} h_\lambda e_\tau,
p_\mu \right>$.  Some of the results we
present in this paper can be generalized to elements
$\het_{(\lambda|\tau)}$ which form a spanning set, but not a basis.
This set of element featured heavily in the thesis of Arash
Islami in a project to develop formulae for
a character basis for the hyperoctahedral group
\cite{Islami}.
\end{remark}

\nocite{*}
\bibliographystyle{amsplain-ac}
\bibliography{properties}
\end{document}